\RequirePackage{fix-cm}
\documentclass[final,smallextended,envcountsame]{svjour3}       
\smartqed  
\usepackage[colorlinks=true,pdfpagemode=UseNone,urlcolor=blue,linkcolor=blue,citecolor=BrickRed,pdfstartview=FitH]{hyperref}

\usepackage{graphicx}
\usepackage{amsmath,amssymb}
\usepackage[usenames,dvipsnames]{xcolor}
\usepackage{mathtools}
\usepackage{microtype}

\usepackage[update]{epstopdf}
\usepackage{todonotes}
\usepackage[ruled,lined,linesnumbered]{algorithm2e}

\usepackage{subfig}

\DeclareMathOperator{\rank}{rank}
\DeclareMathOperator{\tr}{tr}

\DeclareMathOperator{\spann}{span}
\DeclareMathOperator{\ran}{ran}
\DeclareMathOperator*{\minimize}{minimize}

\newcommand{\bbR}{\mathbb{R}}
\newcommand{\abs}[1]{\lvert#1\rvert}

\DeclareMathOperator{\myspan}{{span}}
\DeclareMathOperator{\diag}{{diag}}
\DeclareMathOperator{\mat}{{mat}}

\def\rrank{r}

\renewcommand{\vec}{\operatorname{vec}}
\newcommand{\ignore}[1]{}

\newcommand{\subjectto}{\text{subject to}}

\def\makeheadbox{}

\begin{document}

\title{
Finding a low-rank basis in a matrix subspace
\thanks{This work was supported by JST CREST (Iwata team), JSPS Scientific Research Grants No. 26870149 and No. 26540007, and JSPS Grant-in-Aid for JSPS Fellows No. 267749.}
}


\author{
Yuji Nakatsukasa\and Tasuku Soma\and Andr{\'e} Uschmajew
}


\institute{
           Y. Nakatsukasa \at
           Mathematical Institute \\ University of Oxford, Oxford, OX2 6GG, United Kingdom \\
              \email{yuji.nakatsukasa@maths.ox.ac.uk}           
\and 
T. Soma \at
Graduate School of Information Science \& Technology, \\
            The University of Tokyo,
            7-3-1 Hongo, Bunkyo-ku, Tokyo, Japan \\
     \email{tasuku\_soma@mist.i.u-tokyo.ac.jp}
           \and
           A. Uschmajew \at
              Hausdorff Center for Mathematics \& Institute for Numerical Simulation \\ University of Bonn, 53115 Bonn, Germany\\\email{uschmajew@ins.uni-bonn.de}
}

\date{}

\maketitle

\begin{abstract}
For a given matrix subspace, how can we find a basis that consists of low-rank matrices? 
This is a generalization of the sparse vector problem.
It turns out that when the subspace is spanned by
rank-1 matrices, 
the matrices can be obtained by the tensor CP decomposition. 
For the higher rank case, the situation is not as straightforward. 
In this work we present an algorithm based on a greedy process
applicable to higher rank problems. 
 Our algorithm first estimates the minimum rank by applying soft singular value thresholding 
to a nuclear norm relaxation, and then computes a matrix with that rank using the method of alternating projections. 
We provide local convergence results, and 
compare our algorithm  with several alternative approaches. 
Applications include 
data compression beyond the classical truncated SVD, 
computing accurate eigenvectors of a near-multiple eigenvalue,  
image separation and graph Laplacian eigenproblems. 
\end{abstract}

\section{Introduction}
Finding a succinct representation of a given object has been one of the central tasks in the computer and data sciences.
For vectors, sparsity (i.e., $\ell^0$-norm) is a common measure of succinctness.
For example, exploiting prior knowledge on sparsity of a model is now considered 
crucial in machine learning and statistics~\cite{buhlmann2011statistics}. 
Although the naive penalization of the $\ell^0$-norm easily makes the problem intractable, it turns out that exploiting the $\ell^1$-norm as a regularizer yields a tractable convex problem and performs very well in many settings~\cite{Candes2008,CandesTao2006}.
This phenomenon is strongly related to compressed sensing, which shows
under reasonable assumptions that
 the $\ell^1$-recovery almost always recovers a sparse solution for an undetermined linear system. 
Since matrices are more complex objects, one may consider different criteria on succinctness for matrices, namely the rank.
Interestingly, a variety of concepts from sparse vector recovery 
carry over to low-rank matrix recovery, for which the nuclear norm plays 
a role analogous to the $\ell^1$-norm for sparse vectors~\cite{fazel2002thesis,Fazel01arank}.
The nuclear norm convex relaxation has demonstrated its theoretical and practical usefulness in matrix completion and other low-rank optimization tasks, and is nowadays accompanied by an endless array of optimization algorithms; see, e.g.,~\cite{ames2011nuclear,Cai:SVT,CandesRecht2009,CandesTao2009,LiuSunToh2012,LiuVandenberghe2009,Recht2011}, and~\cite{RechtFazelParillo2010} for a general overview.

Recently, the \emph{sparse vector problem} has been studied by several authors~\cite{Barak2014a,Hand2013,QuSunWright2014,SpielmanWangWright2013}.
In the sparse vector problem, we are given a linear subspace $S$ in $\bbR^n$, and the task is to find the sparsest \emph{nonzero} vector in $S$.
The celebrated results for $\ell^1$-regularization are not directly applicable, and the sparse vector problem is less understood.
The difficulty arises from the nonzero constraints; a natural $\ell^1$-relaxation only yields the trivial zero vector.
Thus the sparse vector problem is nonconvex in its own nature. 
A common approach is based on the \emph{hypercontractivity}, that is, optimizing the ratio of two different norms.
An algorithm that optimizes the $\ell^1 / \ell^\infty$-ratio is studied in \cite{Hand2013,SpielmanWangWright2013}, and the algorithm in~\cite{QuSunWright2014} works with the $\ell^1 / \ell^2$-ratio.
Optimization of the $\ell^4 / \ell^2$-ratio was recently considered by Barak, Kelner, and Steurer~\cite{Barak2014a}.
A closely related problem is the \emph{sparse basis problem}, in which we want to find a basis of $S$ that minimizes the sum of the $\ell^0$-norms. 
In addition to imposing the sparsity of vectors as in the sparse vector problem, a major difficulty here lies in ensuring their linear independence. The recent work~\cite{SparseDictionaryI,SparseDictionaryII} is an extensive theoretical and practical approach to the very similar problem of sparse dictionary learning.

In this paper, we consider the following problem, which we call the \emph{low-rank basis problem}.
Let $\mathcal{M}$ be a matrix subspace in $\bbR^{m \times n}$ of dimension $d$. 
The goal  is to
\begin{equation}
    \begin{alignedat}{2}\label{eq: low-rank basis problem}
        &\minimize  & \quad &\rank(X_1) + \dots + \rank(X_d) \\
        &\subjectto & \quad &\spann\{X_1,\dots,X_d\} = \mathcal{M}.
    \end{alignedat}
\end{equation}
The low-rank basis problem generalizes the sparse basis problem. 
To see this, it suffices to identify $\mathbb{R}^n$ with the subspace of diagonal matrices in $\mathbb{R}^{n \times n}$. 
As a consequence, any result on the low-rank basis problem~\eqref{eq: low-rank basis problem}  (including relaxations and algorithms) will apply to the sparse basis problem with appropriate changes in notation (matrix nuclear norm for diagonal matrices becomes $1$-norm etc.).
Conversely, some known results on the sparse basis problem may generalize to the low-rank basis problem~\eqref{eq: low-rank basis problem}.
An obvious but important example of this logic concerns the complexity of the problem: 
It has been shown in Coleman and Pothen~\cite{ColemanPothen1986} that even if one is given the minimum possible value for $\|x_1\|_0 + \dots + \| x_d\|_0$ in the sparse basis problem, it is NP-hard to find a sparse basis.
Thus the low-rank basis problem~\eqref{eq: low-rank basis problem} is also NP-hard in general.

A closely related (somewhat simpler) problem is the following \emph{low-rank matrix problem}:
\begin{equation}
    \begin{alignedat}{2}\label{eq:low-rank elem problem}
    &\minimize & \quad &\rank (X) \\
    &\subjectto & \quad &X \in \mathcal{M}, X \neq O.
    \end{alignedat}
\end{equation}
This problem is a matrix counterpart of the sparse vector problem.
Again, 
 \eqref{eq:low-rank elem problem} is NP-hard~\cite{ColemanPothen1986}, even if $\mathcal{M}$ is spanned by diagonal matrices.
Note that our problem \eqref{eq:low-rank elem problem} does not fall into the framework of Cai, Cand\'{e}s, and Shen \cite{Cai:SVT}, in which algorithms have been developed for finding a low-rank matrix $X$ in an affine linear space described as $\mathcal{A}(X) = b$ (matrix completion problems are of that type).
In our case we have $b = 0$, but we are of course not looking for the zero matrix, which is a trivial solution for~\eqref{eq:low-rank elem problem}. 
This requires an additional norm constraint, and modifications to the algorithms in~\cite{Cai:SVT}.

\subsection{Our contribution}
We propose an alternating direction algorithm 
for the low-rank basis problem
that does (i) rank estimation, and (ii) obtains a low-rank basis. 
We also provide convergence analysis for our algorithm. 
Our algorithm 
is based on a greedy process, 
whose use we fully justify. 
In each greedy step 
we solve the low-rank matrix problem~\eqref{eq:low-rank elem problem}
in a certain subspace, and hence our algorithm 
can also solve 
the low-rank matrix problem.

Our methods are iterative, and switch between the search of a good low-rank matrix and the projection on the admissible set. The second step typically increases the rank again.
The solution would be a fixed point of such a procedure.
We use two phases with different strategies for the first step, i.e., finding a low-rank matrix.

In the first phase we find new low-rank guesses by applying the singular value shrinkage operator (called \emph{soft thresholding}) considered in Cai, Cand\'{e}s, and Shen~\cite{Cai:SVT}.
In combination with the subspace projection, this results in the matrix analog of the alternating direction algorithm proposed very recently by Qu, Sun, and Wright~\cite{QuSunWright2014} for finding sparse vectors in a subspace.
An additional difficulty, however, arises from the fact that we are required to find 
more than one linearly independent low-rank matrices in the subspace.
Also note that our algorithm adaptively changes the thresholding parameter during its execution, which seems necessary for our matrix problem, although \cite{QuSunWright2014} fixes the thresholding parameter before starting their algorithm.
In our experiments it turns out that the use of the shrinkage operator clearly outperforms alternative operations, e.g., truncating singular values below some threshold, in that it finds matrices with correct rank quite often, but that the distance to the subspace $\mathcal{M}$ is too large.
This is reasonable as the only fixed points of soft thresholding operator are either zero, or, when combined with normalization, matrices with identical nonzero singular values, e.g., rank-one matrices.

As we will treat the subspace constraint as non-negotiable, we will need further improvement.
We replace the shrinkage operator in the second phase by best approximation of the estimated rank (which we call \emph{hard thresholding}). 
Combined with the projection onto the admissible set $\mathcal{M}$, this then delivers low-rank matrices in the subspace $\mathcal{M}$ astonishingly reliably (as we shall see, this second phase is typically not  needed when a rank-one basis exists).


Our convergence analysis in Section~\ref{sec:convanal} provides further insights into the behavior of the process, in particular the second phase. 

\subsection{Applications}
The authors were originally motivated to 
consider the low-rank basis problem 
by applications in discrete mathematics~\cite{Gurvits2004,Ivanyos2014,Lovasz1989}. 
It can be useful also in various other settings, some of which we outline below. 
\subsubsection*{Compression of SVD matrices}
Low-rank matrices arise frequently in applications 
and a low-rank (approximate) decomposition 
such as the SVD is often used to 
reduce the storage to represent the matrix $A\in\mathbb{R}^{m\times n}$: $A \approx U\Sigma V^T$.
Here $\Sigma\in\mathbb{R}^{r\times r}$ where $r$ is the rank. 
The memory requirement for storing the whole $A$ is clearly $mn$, whereas 
$U,\Sigma,V$ altogether require $(m+n)r$ memory (we can dismiss $\Sigma$ by merging it into $V$). 
Hence,  the storage reduction factor is
\begin{equation}  \label{eq:compini}
\frac{(m+n)r}{mn},   
\end{equation}
so if the rank $r$ is much smaller than $\min(m,n)$ then we achieve significant memory savings. 

This is all well known, but here we go one step further and try to reduce the memory cost for representing the matrices $U,V$.
Note that the same idea of using a low-rank representation is useless here, as these matrices have orthonormal columns and hence the singular values are all ones. 

The idea is the following: if we \emph{matricize} the columns of $U$ (or $V$), those matrices might have a low-rank structure. More commonly, there might exist a nonsingular matrix $W\in\mathbb{R}^{r\times r}$ such that the columns of $UW$ have low-rank structure when matricized. 
We shall see that many orthogonal matrices that arise in practice have this property. 
The question is, then, how do we find such $W$ and the resulting compressed representation of $U$? 
This problem boils down to the low-rank basis problem, in which $\mathcal{M}$ is  the linear subspace spanned by the matricized columns of $U$.

To simplify the discussion here we assume $m=s^2$ for an integer $s$ (otherwise, e.g. when $m$ is prime, a remedy is to pad zeros to the bottom). 
Once we find an appropriate $W$ for $U\in\mathbb{R}^{s^2\times r}$, we represent the matricization of each column as a low-rank (rank $\hat r$) matrix $U_{\hat r}\hat\Sigma V_{\hat r}$, which is represented using $2s\hat r$ memory, totaling to $2sr\hat r+r^2$ where $r^2$ is for $W$. 
Since the original $U$ requires $s^2r$ memory with $r\ll s^2$, this can significantly reduce the storage if $\hat r\ll r$. 

When this is employed for both $U$ and $V$ the overall storage reduction for representing $A$ becomes 
\begin{equation}
  \label{eq:compressrate}
\frac{4sr\hat r+r^2}{mn}  . 
\end{equation}
For example, when $m=n$, $r = \delta m$ and $\hat r=\hat\delta s $ for $\delta,\delta \ll 1$ 
this factor is 
\begin{equation}
  \label{eq:compressrateest}
4\delta\hat \delta+\delta^2, 
\end{equation}
achieving a ``squared'' reduction factor compared with \eqref{eq:compini}, which is about $2\delta$. 

Of course, we can further reduce the memory by recursively matricizing the columns of $U_{\hat r}$, as long as it results in data compression.

\subsubsection*{Computing and compressing an exact eigenvector of a multiple eigenvalue}\label{sec:introeigvec}
Eigenvectors of a multiple eigenvalue are not unique, 
and those corresponding to near-multiple eigenvalues generally cannot be computed to high accuracy. 
We shall show that it is nonetheless sometimes possible to compute exact eigenvectors of  near-multiple eigenvalues, if additional 
property is present that the eigenvectors are low-rank when matricized. 
This comes with the additional benefit of storage reduction, as discussed above.
We describe more details and experiments in  Section~\ref{sec:exeig}.

\subsubsection*{The rank-one basis problem}
An interesting and important subcase of the low-rank basis problem is the \emph{rank-one basis problem}; in this problem, we are further promised that a given subspace $\mathcal{M}$ is spanned by rank-one matrices.
Gurvits~\cite{Gurvits2004} first considered the rank-one basis problem in his work on the \emph{Edmonds problem}~\cite{Edmonds1967}.
Below let us explain his original motivation and connections to combinatorial optimization.

The task of Edmonds' problem is to find a 
matrix of \emph{maximum} rank in a given matrix subspace.
It is known that sampling a random element in a matrix subspace yields a solution with high probability, and therefore the main interest is to design \emph{deterministic} algorithms
(of course, this motivation is closely related to \emph{polynomial identity testing} and the P vs BPP conjecture~\cite{motwani2010randomized}).
For some special cases, one can devise a deterministic algorithm by exploiting combinatorial structure of a given subspace. 
In particular, if a given basis consists of 
rank-one matrices (which are known), Edmonds' problem reduces to linear matroid intersection~\cite{Lovasz1989}, which immediately implies the existence of deterministic algorithms.

Gurvits~\cite{Gurvits2004} studied a slightly more general setting: a given subspace is only promised to be spanned by some rank-one matrices, which are unknown.
If one can solve the rank-one basis problem, this setting reduces to the previous case using the solution of rank-one basis problem.
Indeed, he conjectured that the rank-one basis problem is NP-hard, and designed a deterministic algorithm for his rank maximization problem without finding these rank-one matrices explicitly.
We also note that Ivanyos~et~al.~\cite{Ivanyos2014} investigated the Edmonds problem for a matrix subspace on finite fields, which is useful for \emph{max-rank matrix completion} (see \cite{HKM05,Harvey2006} and references therein).

\subsubsection*{Tensor decomposition}


The rank-one basis problem has an interesting connection to tensor decomposition: finding a rank-one basis for a $d$-dimensional matrix subspace amounts to finding a \emph{CP decomposition}~(e.g., \cite[Sec.~3]{KoldaBader2009}) of representation rank $d$ for a third-order tensor with slices $M_1,\dots,M_d$ that form a basis for $\mathcal{M}$.
For the latter task very efficient nonconvex optimization algorithm like alternating least squares exist, which, however, typically come without any convergence certificates. 
An alternative, less cheap, but exact method uses simultaneous diagonalization, which are applicable when $d\leq \min(m,n)$.
Applying these methods will often be successful when a rank-one basis exists, but fails if not.
This tensor approach seems to have been overseen in the discrete optimization community so far, and we explain it in Appendix~\ref{sec: reuction to tensors}.

Even in general, when no rank-one basis exists, the representation of matrices $M_1,\dots,M_d$ in a low-rank basis can be seen as an instance of \emph{tensor block term decomposition}~\cite{DeLathauwer2008b} with $d$ blocks. Hence, given any tensor with slices $M_1,\dots,M_d$ (not necessarily linearly independent), our method may be used to obtain a certain block term decomposition with low-rank matrices constrained to the span of $M_1,\dots,M_d$. These matrices may then be further decomposed into lower-rank components. In any case, the resulting sum of ranks of the basis provides an upper bound for the CP rank of the tensor under consideration.

While in this paper we do not dwell on the potential applications of our method to specific tensor decomposition tasks as they arise in signal processing and data analysis (see~\cite{cichocki2,KoldaBader2009} for overview), we conduct a numerical comparison with a state-of-the-art tensor algorithm for CP decomposition for synthetic rank-one basis problems; see Section~\ref{sec:comparetensor}.


\subsection{Outline and notation}
The rest of this paper is organized as follows.
Section~\ref{sec:greedy} proves that a greedy algorithm would solve 
the low-rank basis problem, if each greedy step 
(which here is NP-hard)
is successful. 
In Section \ref{sec:alg}, we present our algorithm for the low-rank basis problem that follows the greedy approach.
We then analyze convergence of phase II in our algorithm in Section \ref{sec:convanal}.
Experimental evaluation of our algorithm is illustrated in Section~\ref{sec: experiments}. 
For the special case of the rank-one basis problem, we describe the alternative approach via tensor decomposition in Appendix~\ref{sec: reuction to tensors}. 

\paragraph{Notation} We summarize our notation:
$m\times n$ is the size of the matrices in $\mathcal{M}$; $d$  is the  dimension of the subspace $\mathcal{M} \subseteq \mathbb{R}^{m \times n}$;
$\rrank$ will denote a rank. We use the notation 
$\mat(x)\in\mathbb{R}^{m\times n}$ for the matricization of a vector $x\in\mathbb{R}^{mn}$, and 
$\vec(X)\in\mathbb{R}^{mn}$ denotes the inverse operation for $X \in \mathbb{R}^{m\times n}$. 

%
%
%

\section{The abstract greedy algorithm for the low-rank basis problem}\label{sec:greedy}

As already mentioned, the low-rank basis problem~\eqref{eq: low-rank basis problem} for a matrix subspace $\mathcal{M}$ is a generalization of the sparse basis problem for subspaces of $\mathbb{R}^n$.
In~\cite{ColemanPothen1986} it was shown that a solution to the sparse basis problem can be in principle found using a greedy strategy.
The same is true for~\eqref{eq: low-rank basis problem}, as we will show next. The corresponding greedy algorithm is given as Algorithm~\ref{alg:greedy algorithm}. 
Indeed, this algorithm can be understood as a greedy algorithm for an \emph{infinite matroid}~\cite{Oxley1992} of finite rank.
We can prove that Algorithm~\ref{alg:greedy algorithm} finds a minimizer of \eqref{eq: low-rank basis problem}, by adapting a standard proof for greedy algorithms on \emph{finite} matroids.
Note that this fact does not imply that \eqref{eq: low-rank basis problem} is tractable, since finding $X_\ell^*$ in the algorithm is a nontrivial task.

\begin{algorithm}[h!]
\linespread{1.2}\selectfont
\DontPrintSemicolon
\KwIn{Subspace $\mathcal{M} \subseteq \mathbb{R}^{m \times n}$ of dimension $d$.}
\KwOut{Basis $\mathcal{B} = \{ X_1^*,\dots,X_d^*\}$ of $\mathcal{M}$.}
Initialize $\mathcal{B} = \emptyset$.\;
\For{$\ell = 1,\dots,d$}{
	Find $X_\ell^* \in \mathcal{M}$ of lowest possible rank such that $\mathcal{B} \cup \{ X_\ell^*\}$ is linearly independent.\;
	$\mathcal{B} \leftarrow \mathcal{B} \cup \{ X_\ell^* \}$\;
}
\caption{Greedy meta-algorithm for computing a low-rank basis}\label{alg:greedy algorithm}
\end{algorithm}


\begin{lemma}
Let $X_1^*,\dots,X_d^*$ denote matrices constructed by the greedy Algorithm~\ref{alg:greedy algorithm}. Then for any $1 \le \ell \le d$ and linearly independent set $\{ X_1,\dots,X_\ell\} \subseteq \mathcal{M}$ with $\rank(X_1) \le \dots \le \rank(X_\ell)$, it holds
\[
\rank(X_i) \ge \rank(X_i^*) \quad \text{for $i=1,\dots,\ell$.}
\]
\end{lemma}
\begin{proof}
The proof is by induction. By the greedy property, $X_1^*$ is a (nonzero) matrix of minimal possible rank in $\mathcal{M}$, i.e., $\rank(X_1^*) \le \rank(X_1)$. For $\ell >1$, if $\rank(X_\ell) < \rank(X_\ell^*)$, then $\rank(X_i) < \rank(X_\ell^*)$ for all $i = 1,\dots, \ell$. But since one $X_i$ must be linearly independent from $X_1^*,\dots,X_{\ell-1}^*$, this would contradict the choice of $X^*_\ell$ in the greedy algorithm.
\qed\end{proof}

We say a linearly independent set $\mathcal{B}_\ell = \{ \hat{X}_1,\dots,\hat{X}_\ell \} \subseteq \mathcal{M}$ is of \emph{minimal rank}, if
\[
 \sum_{i=1}^\ell \rank(\hat{X}_i) = \min \left\{ \sum_{i=1}^\ell \rank(X_i)  \vcentcolon \text{$\{X_1,\dots,X_\ell\} \subseteq \mathcal{M}$ is linearly independent} \right\}.
\]
The following theorem is immediate from the previous lemma.

\begin{theorem}
Let $X_1^*,\dots,X_d^*$ denote matrices constructed by the greedy Algorithm~\ref{alg:greedy algorithm}, and let $\mathcal{B}_\ell = \{X_1,\dots,X_\ell\} \subseteq \mathcal{M}$ be a linearly independent set with $\rank(X_1) \leq \dots \leq \rank(X_\ell)$. Then $\mathcal{B}_\ell$ is of minimal rank if (and hence only if) 
\[
\rank(X_i) = \rank(X_i^*) \quad \text{for $i=1,\dots,\ell$.}
\]
In particular, $\{ X_1^* ,\dots,X_\ell^* \}$ is of minimal rank.
\end{theorem}



A remarkable corollary is that the ranks of the elements in a basis of lowest rank are essentially unique.

\begin{corollary}
The output $\mathcal{B} = \{X_1^*,\dots,X_d^*\}$ of Algorithm~\ref{alg:greedy algorithm} solves the low-rank basis problem~\eqref{eq: low-rank basis problem}, that is, provides a basis for $\mathcal{M}$ of lowest possible rank. Any other basis of lowest rank takes the same ranks $\rank(X_\ell^*)$ up to permutation.
\end{corollary}

It is worth mentioning that even for the analogous sparse basis problem our results are stronger than Theorem 2.1 in~\cite{ColemanPothen1986} (which only states that $\mathcal{B}$ will be a sparsest basis). Moreover, our proof is different and considerably simpler. We are unaware whether the above results on the particular low-rank basis problem have been observed previously in this simple way.

\section{Finding low-rank bases via thresholding and projection}\label{sec:alg}

In this main section of this article, we propose an algorithm that tries to execute the abstract greedy Algorithm~\ref{alg:greedy algorithm} using iterative methods on relaxed formulations. 


The greedy algorithm suggests finding the matrices $X_1,\dots,X_d$ one after another, during which we
monitor the linear dependence when computing $X_\ell$ with respect to the previously computed $X_1,\dots,X_{\ell-1}$, and apply some restart procedure when necessary. Alternatively, one can try to find low-rank matrices $X_1,\dots,X_d \in \mathcal{M}$ in parallel, monitor their linear independence, and reinitialize the ones with largest current ranks in case the basis becomes close to linearly dependent. In both cases, the linear independence constraint, which substantially increases the hardness of the problem, is in principle ignored as long as possible, and shifted into a restart procedure. 
Therefore, we mainly focus on iterative methods to solve the problem~\eqref{eq:low-rank elem problem} of finding a single low-rank matrix $X$ in $\mathcal{M}$. The complete procedure for the low-rank basis problem will be given afterwards in Section~\ref{subsec: the full algorithm}. 


The first algorithm we consider for solving the low-rank basis problem~\eqref{eq:low-rank elem problem} alternates between soft singular value thresholding (shrinkage) and projection onto the subspace $\mathcal{M}$, and will be presented in the next subsection.
During our work on this paper, an analogous method for the corresponding sparse vector problem of minimizing $\| x \|_0$ over $x \in S, \ \| x\|_2 = 1$ has been independently derived and called a ``nonconvex alternating direction method (ADM)'' for a modified problem in the very recent publication~\cite{QuSunWright2014}.
This reference also contains a motivation for using the Euclidean norm for normalization.
We have decided to adopt this derivation, but will use the terminology \emph{block coordinate descent} (BCD) instead, which seems more in line with standard terminology regarding the actual update rules. 
However, as it turns out, this algorithm alone indeed provides good rank estimators, but very poor subspace representations.
This is very understandable when the target rank is higher than one, since the fixed points of the singular value shrinkage operator explained below are matrices whose positive singular values are all equal, which do not happen in generic cases.
Therefore, we turn to a second phase that uses hard singular value thresholding (rank truncation) for further improvement, as will be explained in Section~\ref{sec:algconv}.
%

\subsection{Phase I: Estimating a single low-rank matrix via soft thresholding}\label{sec:algrank}

The starting point is a further relaxation of
~\eqref{eq:low-rank elem problem}: the rank function, that is, the number of nonzero singular values, is replaced by the matrix nuclear norm $\| X \|_*$, which equals the sum of singular values.
This leads to the problem
\begin{equation}\label{eq: considered problem for one matrix relaxed}
\begin{alignedat}{2}
&\minimize & \quad &\| X \|_*\\
&\subjectto & \quad &X \in \mathcal{M}, \text{ and } \ \| X \|_F = 1.
\end{alignedat}
\end{equation}
The relaxation from rank to nuclear norm can be motivated by the fact that in case a rank-one solution exists, it will certainly be recovered by solving~\eqref{eq: considered problem for one matrix relaxed}. 
For higher ranks, it is less clear under which circumstances the nuclear norm provides an exact surrogate for the rank function under the given spherical constraint.
For an example of a space $\mathcal{M}$ spanned by matrices of rank at most $r$ for which the minimum in~\eqref{eq: considered problem for one matrix relaxed} is attained at a matrix of rank larger than $r$, 
consider $M_1 = \mbox{diag}(1,-\sqrt{\epsilon},-\sqrt{\epsilon},0,0,0,0)$, $M_2 = 
\mbox{diag}(0,1,0,\sqrt{\epsilon},\sqrt{\epsilon},0,0)$, and $
M_3 = \mbox{diag}(0,0,1,0,0,\sqrt{\epsilon},\sqrt{\epsilon})$. 
Every linear combination involving at least two matrices has at least rank four.
So the $M_i$ are the only matrices in their span of rank at most three.
After normalization w.r.t. the Frobenius norm, their nuclear norm equals $\|M_i\|_*/\|M_i\|_F=(1+2\sqrt{\epsilon})/\sqrt{1+2\epsilon}=1+2\sqrt{\epsilon}+O(\epsilon)$.
But for $\epsilon$ small enough, the rank five matrix $X=M_1-\sqrt{\epsilon}M_2-\sqrt{\epsilon}M_3=(1,0,0,\epsilon,\epsilon,\epsilon,\epsilon)$ has a smaller nuclear norm $\|X\|_*/\|X\|_F=(1+4\epsilon)/\sqrt{1+4\epsilon^2}=1+4\epsilon+O(\epsilon^2)$ after normalization.

\subsubsection*{Soft thresholding and block coordinate descent}

Nevertheless, 
such counterexamples are rather contrived, and
we consider~\eqref{eq: considered problem for one matrix relaxed}  a good surrogate for
~\eqref{eq:low-rank elem problem} in the generic case. 
The problem is still very challenging due to the non-convex constraint. In~\cite{QuSunWright2014} a block coordinate descent (BCD) method has been proposed to minimize the $\ell^1$-norm of a vector on an Euclidean sphere in a subspace of $\mathbb{R}^n$. As we explained above, this problem is a special case of~\eqref{eq: considered problem for one matrix relaxed}, and the algorithm can be generalized as follows.

%


 Given a current guess $X \in \mathcal{M}$ we are looking for a matrix $Y$ of lower-rank in a neighborhood if possible. For this task, we use the \emph{singular value shrinkage operator} $\mathcal{S}_\tau$~\cite{Cai:SVT}: letting $X = U\Sigma V^T$ be a singular value decomposition with $\Sigma = \diag(\sigma_1,\dots,\sigma_{\rank(X)})$, $\sigma_1 \ge \dots \ge \sigma_{\rank(X)} > 0$, we choose $Y$ as
\begin{equation}\label{eq: shrinkage operator}
 Y = \mathcal{S}_\tau(X) = U  \mathcal{S}_\tau(\Sigma)V^T,\quad \mathcal{S}_\tau(\Sigma) = \diag(\sigma_1-\tau,\dots,\sigma_{\rank(X)} - \tau)_+,
\end{equation}
where $x_+ :=\max(x,0)$ and $\tau > 0$ is a thresholding parameter. The rank of $Y$ now equals the number of singular values $\sigma_i$ larger than $\tau$. Note that even if the rank is not reduced the ratios of the singular values increase, since $(\sigma_i - \tau)/(\sigma_j - \tau) > \sigma_i/\sigma_j$ for all 
$(i,j)$ such that $\tau <\sigma_j<\sigma_i$. Hence a successive application of the shift operator will eventually remove all but the dominant singular value(s), even if the iterates are normalized in between (without in-between normalization it of course removes all singular values). This effect is not guaranteed when simply deleting singular values below the threshold $\tau$ without shifting the others, as it would preserve the ratios of the remaining singular values, and might result in no change at all if $\tau$ is too small. But even if this \emph{hard threshold} was chosen such that at least one singular value is always removed, we found 
through experiments
that this does not work as well in combination with projections onto a subspace $\mathcal{M}$ as the soft thresholding.

The new matrix $Y$ in~\eqref{eq: shrinkage operator} will typically not lie in $\mathcal{M}$, nor will it be normalized w.r.t. the Frobenius norm. 
Thus, introducing the orthogonal projector (w.r.t. the Frobenius inner product) $\mathcal{P}_{\mathcal{M}}$ from $\mathbb{R}^{m \times n}$ onto $\mathcal{M}$, which is available given \emph{some} basis $M_1,\dots,M_d$ of $\mathcal{M}$, we consider the fixed point iteration:
\begin{equation}\label{eq: update formulas}
Y = \mathcal{S}_{\tau}(X), \quad
X = \frac{\mathcal{P}_{\mathcal{M}}(Y)}{\|\mathcal{P}_{\mathcal{M}}(Y)\|_F}.
\end{equation}

The projection $\mathcal{P}_{\mathcal{M}}$ is precomputed at the beginning and defined as $\mathbf{M}\mathbf{M}^T$ (only $\mathbf{M}$ is stored), 
where $\mathbf{M}$ is the orthogonal factor of the thin QR decomposition~\cite[Sec.~5]{golubbook4th} of the matrix 
\(
[\vec(M_1), \dots, \vec(M_d)] \in \mathbb{R}^{mn \times d}
\), where the (not necessarily low-rank) matrices $M_i$ span $\mathcal{M}$. 
It is used in the following way:
\begin{equation}\label{eq: projection}
\mathcal{P}_{\mathcal{M}}(Y) = \mat(\mathbf{M}\mathbf{M}^T\vec(Y)),
\end{equation}

To obtain the interpretation as BCD, we recall the fact that the shrinkage operation provides the unique solution to the strictly convex problem
\begin{equation}\label{eq:svtopt}
\minimize_Y \quad \tau \|Y\|_*^2 + \frac{1}{2}\|Y - X\|_F^2,
\end{equation}
see~\cite{Cai:SVT}. Intuitively,~\eqref{eq:svtopt} attempts to solve~\eqref{eq: considered problem for one matrix relaxed} in a neighborhood of the current guess $X$, while ignoring the constraints. The parameter $\tau$ controls the balance between small nuclear norm and locality: the larger it is the lower $\rank(Y)$ becomes, but $Y$ will be farther from $X$. Taking $\tau$ small has the opposite effect. The explicit form~\eqref{eq: shrinkage operator} quantifies this qualitative statement, as the distance between $X$ and $Y$ is calculated as
\[
\|X - Y \|_F^2 = \sum_{\sigma_i > \tau} \tau^2 + \sum_{\sigma_i \le \tau} \sigma_i^2.
\]

As a result, we see that the formulas~\eqref{eq: update formulas} represent the update rules when applying BCD to the problem
\begin{equation*}
\begin{alignedat}{2}
&\minimize_{X,Y} & \quad &\tau \| Y \|_* + \frac{1}{2}\|Y - X\|_F^2\\
&\subjectto & \quad &X \in \mathcal{M}, \text{ and } \ \| X \|_F = 1,
\end{alignedat}
\end{equation*}
which can be seen as a penalty approach to approximately solving~\eqref{eq: considered problem for one matrix relaxed}. 
%

\subsubsection*{Algorithm with adaptive shift $\tau$}

The considerations are summarized as Algorithm~\ref{alg:rankest}. At its core it is more or less analogous to the algorithm in~\cite{QuSunWright2014}. However, a new feature is that the parameter $\tau$ is chosen adaptively in every iteration.

\begin{algorithm}
\linespread{1.2}\selectfont
\DontPrintSemicolon
\KwIn{Orthogonal projection $\mathcal{P}_\mathcal{M}$ on $\mathcal{M}$; scalars $\delta, maxit, changeit > 0$, $\tau_{tol} \ge 0$; initial guess $X \in \mathcal{M}$ with $\|X\|_F=1$, 
initialize $r = n$.  
}
\KwOut{$X$, $Y$, and $r$, where $X = \mathcal{P}_{\mathcal{M}}(Y) \in \mathcal{M}$, and $Y$ is a matrix of low rank $r$ which is close to or in $\mathcal{M}$.}
\For{$it = 1, \dots, maxit$}{
		$X = U \Sigma^{} V^T$ \tcp*[1]{singular value decomposition}
		$s = \abs{\{ \sigma_i \vcentcolon \sigma_i > \tau_{tol}\}}$ \tcp*[1]{`noiseless' rank}\label{algline:essential rank}
		$X \leftarrow \mathcal{T}_s(X) / \| \mathcal{T}_s(X) \|_F$  \tcp*[1]{remove `noisy' singular values}\label{algline:essential rank 2}
		$\tau = \delta / \sqrt{s}$\tcp*[1]{set singular value shift}\label{algline:tau}
		$Y \leftarrow \mathcal{S}_\tau(X)$\tcp*[1]{shrinkage}
		$r \leftarrow \min(r,\rank(Y))$\tcp*[1]{new rank estimate}
		$X \leftarrow \mathcal{P}_{\mathcal{M}}(Y)$\tcp*[1]{projection onto subspace}
		$X \leftarrow X/\|X\|_F$\tcp*[1]{normalization}
	   Terminate if $r$ has not changed for $changeit$ iterations.\;
}
\caption{Phase I -- Rank estimation}\label{alg:rankest}
\end{algorithm}

The choice of the singular value shift $\tau$ in line \ref{algline:tau} is made to achieve faster progress, and motivated by the fact that a matrix $X$ of Frobenius norm 1 has at least one singular value below and one above $1/\sqrt{\rank(X)}$, unless all singular values are the same. Therefore, the choice of $\tau = 1/\sqrt{\rank(X)}$ would always remove at least one singular value, but can also remove all but the largest singular value in case the latter is very dominant. To make the shift more efficient in case that many small singular values are present, we rely instead on an effective rank $s$ which is obtained by discarding all singular values below a minimum threshold $\tau_{tol}$, considered as noise (line~\ref{algline:essential rank}). On the other hand, since the sought low-rank matrix may happen to have clustered dominant singular values, it is important to choose a less aggressive shift by multiplying with $0<\delta<1$. The choice of the parameters $\tau_{tol}$ and $\delta$ is heuristic, and should depend on the expected magnitudes and distribution of singular values for the low rank bases. In most of our experiments we used $\delta = 0.1$ and $\tau_{tol} = 10^{-3}$. With such a small value of $\tau_{tol}$ (compared to $\delta$) the truncation and normalization in line~\ref{algline:essential rank 2} are optional and make only a subtle difference.


Of course, one may consider alternative adaptive strategies such as $\tau \sim \| X \|_*/\rank(X)$ or $\tau \sim \sigma_{\rank(X)}$. Also $\tau_{tol}$ and $\delta$ may be set adaptively.

\subsubsection*{Choice of the initial guess}

Let us remark on the choice of the initial guess. As we shall see later and can be easily guessed, with randomly generated initial 
$X \in \mathcal{M}$, the output $r$ is not always the rank of the lowest-rank matrix in $\mathcal{M}$. 
A simple way to improve the rank estimate is to repeat Phase I with several initial matrices, and adopt the one that results in the smallest rank. 
Another ``good'' initial guess seems to be the analogue of that suggested in 
\cite{QuSunWright2014}, 
but this is intractable because the analogue here would be to initialize with a projection onto every possible rank-one matrix, 
and there are clearly infinitely many choices.
We therefore mainly employ random initial guesses, and 
leave the issue of finding a good initial guess an open problem.

\subsubsection*{The use as a rank estimator}

In our experiments we observed that Algorithm~\ref{alg:rankest} alone is often not capable of finding a low-rank matrix in the subspace. Typically the two subsequences for $Y$ and $X$ in~\eqref{eq: update formulas} produce two different numerical limits: $Y$ tends to a low-rank matrix which is close to, but not in the subspace $\mathcal{M}$; by contrast, the $X$ are always in the subspace, but are typically not observed to converge to a low-rank matrix. In fact, we can only have $X = Y$ in~\eqref{eq: update formulas} for rank-one matrices, in addition to the special case where 
the rank is higher but the singular values are all equal. 
Therefore, in the general case, further improvement will be necessary (phase II below). However, as it turns out, the rank found by the sequence of $Y$ provides a surprisingly good estimate also for the sought minimal rank in the subspace. 
Moreover, the obtained $X$ also provides the starting guess $X = \mathcal{P}_{\mathcal{M}}$ for further improvement in the second phase, described next. 
An analogous statement was proven in~\cite{QuSunWright2014full} for the sparse vector problem (which can be regarded as a special case of ours), but the analysis there assumes the existence of a sparse vector in a subspace of otherwise random vectors; here we do not have such (or related) assumptions.
In Section~\ref{subsec:convanal} we give some qualitative explanation for why we expect this process to obtain the correct rank, but 
we leave a detailed and rigorous analysis of Algorithm~\ref{alg:rankest} an open problem and call this preliminary procedure ``Rank estimation''.


\subsection{Phase II: Finding a matrix of estimated rank via alternating projections}\label{sec:algconv}

We now turn to the second phase, in which we find the matrix $X\in\mathcal{M}$ such that $\rank(X) = r$, the output rank of Phase I. Essentially, the soft singular value thresholding in Phase I is replaced by hard thresholding.

\subsubsection*{Alternating projections}

In Phase II of our algorithm we assume that we know a rank $r$ such that $\mathcal{M}$ contains a matrix of rank at most $r$. To find that matrix, we use the method of alternating projections between the Euclidean (Frobenius) unit sphere in the subspace $\mathcal{M}$ and the closed cone of matrices of rank at most $r$. The metric projection (in Frobenius norm) on this cone is given by the \emph{singular value truncation operator} $\mathcal{T}_r$ defined as
\[
 \mathcal{T}_r(X) = U \mathcal{T}_r(\Sigma)V^T,\quad \mathcal{T}_r(\Sigma) = \diag(\sigma_1,\dots,\sigma_r,0,\dots,0),
\]
where $X = U\Sigma V^T$ is an SVD of $X$ with $\Sigma = \diag(\sigma_1,\dots,\sigma_{\min(m,n)})$. The method of alternating projections hence reads
\begin{equation}\label{eq: update formulas alternating projection}
Y = \mathcal{T}_{r}(X), \quad
X = \frac{\mathcal{P}_{\mathcal{M}}(Y)}{\|\mathcal{P}_{\mathcal{M}}(Y)\|_F}.
\end{equation}
Conceptually, this iteration is the same as~\eqref{eq: update formulas} with the soft thresholding operator $\mathcal{S}_\tau$ replaced by the hard thresholding operator $\mathcal{T}_r$. Alternatively,~\eqref{eq: update formulas alternating projection} can be interpreted as employing BCD for the problem
\begin{equation*}
\begin{alignedat}{2}
&\minimize_{X,Y} & \quad &\|Y - X\|_F\\
&\subjectto & \quad &X \in \mathcal{M}, \ \| X \|_F = 1, \text{ and } \rank(Y) \le r.
\end{alignedat}
\end{equation*}

As a result, we obtain Algorithm~\ref{alg:getmats}.

\begin{algorithm}[h]
\DontPrintSemicolon
\KwIn{Orthogonal projection $\mathcal{P}_\mathcal{M}$ on $\mathcal{M}$; rank $r \ge 1$; scalars $maxit, tol > 0$; initial guess $X \in \mathcal{M}$.}
\KwOut{$X$, $Y$, where $X = \mathcal{P}_{\mathcal{M}}(Y) \in \mathcal{M}$, and $Y$ is a matrix of rank $r$, and hopefully $\| X - Y \|_F \le tol$.}
\While{$\| X - Y \|_F > tol$ {\upshape and} $it \le maxit$}{
	$X = U \Sigma V^T$ \tcp*[1]{singular value decomposition}
	$Y \leftarrow \mathcal{T}_r(X)$\label{alg:getmats:bestlowrankapprox} \tcp*[1]{best rank-$r$ approximation}
	$X \leftarrow \mathcal{P}_{\mathcal{M}}(Y)\label{alg:getmats:project}$\tcp*[1]{projection onto subspace}
	$X \leftarrow X/\|X\|_F$\label{alg:getmats:normalize}\tcp*[1]{normalization}
}
\caption{Phase II -- Alternating projections}\label{alg:getmats}
\end{algorithm} 

 The authors of the aforementioned reference~\cite{QuSunWright2014}, who proposed Algorithm~\ref{alg:rankest} for the sparse vector problem, also suggest a second phase (called ``rounding''), which is, however, vastly different from our Algorithm~\ref{alg:getmats}. It is based on linear programming and its natural matrix analogue would be to solve
\begin{equation}  \label{eq:minLP}
\begin{alignedat}{2}
&\minimize & \quad &\|X\|_*\\
&\subjectto & \quad &X \in \mathcal{M}, \text{ and } \tr(\tilde{X}^TX) = 1.
\end{alignedat}
\end{equation}
Here $\tilde{X}$ is the final matrix $X$ from Algorithm~\ref{alg:rankest}. In~\cite{QuSunWright2014,QuSunWright2014full} it is shown for the vector case that if $\tilde{X}$ is sufficiently close to a global solution of~\eqref{eq: considered problem for one matrix relaxed} then we can recover it exactly by solving~\eqref{eq:minLP}.

\subsubsection*{Comparison with a convex optimization solver}
Note that unlike our original problem~\eqref{eq: low-rank basis problem} or its nuclear norm relaxation~\eqref{eq: considered problem for one matrix relaxed},~\eqref{eq:minLP} is a convex optimization problem, since the constraints are now the linear constraint 
$\tr(\tilde{X}^TX) = 1$ along with the restriction in the subspace $X \in \mathcal{M}$. Nuclear norm minimization under linear constraints has been intensively considered in the literature, see~\cite{Cai:SVT,RechtFazelParillo2010} and references therein for seminal work. A natural approach is to attempt to solve \eqref{eq:minLP} 
by some convex optimization solver. 

In view of this, we conduct experiments to compare 
our algorithm with one where Phase II is replaced by the cvx optimization code~\cite{cvx}. 
For the test we used the default tolerance parameters in cvx. 
We vary $n$  while taking $m=n$ and generated 
the matrix subspace $\mathcal{M}$ so that the exact ranks are all equal to five. 
Here and throughout, the numerical experiments were carried out 
in MATLAB version R2014a on a desktop machine with an Intel Core i7 processor 
and 16GB RAM. 

The runtime and accuracy are shown in Table~\ref{tab:cvx}. 
Here the accuracy is measured as follows: letting $\hat X_i$ for $i=1,\ldots,d$ be the computed rank-1 matrices, we form the $mn\times d$ matrix $\mathcal{\hat X} = [\vec(\hat X_1),\ldots,\vec(\hat X_d)]$, and compute the error as the subspace angle~\cite[Sec.~6.4.3]{golubbook4th} between $\mathcal{\hat X}$ and the exact subspace $[\vec(M_1),\ldots,\vec(M_d)]$. 
Observe that 
while both algorithms provide (approximate) desired solutions, 
Algorithm~\ref{alg:heuristic greedy algorithm}   is more accurate, and much faster with the difference in speed increasing rapidly with the matrix size (this is for a rough comparison purpose: cvx is not very optimized for nuclear norm minimization).

\begin{table}[htbp]
  \begin{center}
    \caption{Comparison between Alg.~\ref{alg:heuristic greedy algorithm} and Phase II replaced with cvx.}
\label{tab:cvx}
\subfloat[Runtime(s).]{
\begin{tabular}{c|cccc}
$n$  &20&30&40 &50 \\\hline
Alg.~\ref{alg:heuristic greedy algorithm}  
&   1.4&  2.21&  3.38 & 4.97 \\
{\tt cvx}&  28.2&   186&   866 & 2960 \\
\end{tabular}
}
\subfloat[Error]{
\begin{tabular}{c|cccc}
 $n$ &20&30&40&50 \\\hline
Alg.~\ref{alg:heuristic greedy algorithm}  
 &2.9e-15&8.0e-15&9.5e-15& 2.1e-14\\
{\tt cvx} &6.4e-09&5.2e-10&5.7e-10 & 2.8e-10
\end{tabular}
}
  \end{center}
\end{table}

Another approach to solving \eqref{eq:minLP} is Uzawa's algorithm as described in~\cite{Cai:SVT}.
However, our experiments suggest that 
Uzawa's algorithm gives poor accuracy (in the order of magnitude $10^{-4}$), 
especially when the rank is not one. 

In view of these observations, in what follows we do not consider a general-purpose solver for convex optimization 
and focus on using Algorithm~\ref{alg:getmats} for Phase II.

\subsection{A greedy algorithm for the low-rank basis problem}\label{subsec: the full algorithm}

\subsubsection*{Restart for linear independence}

Algorithms~\ref{alg:rankest} and~\ref{alg:getmats} combined often finds a low-rank matrix in $\mathcal{M}$. To mimic the abstract greedy Algorithm~\ref{alg:greedy algorithm}, this can now be done consecutively for $\ell =1, \dots, d$. However, to ensure linear independence 
among the computed matrices $X_i$, 
a restart procedure may be necessary. 
After having calculated $X_1,\dots,X_{\ell-1}$ and ensured that they are linearly independent, the orthogonal projector $\mathcal{P}_{\ell-1}$ onto $\spann\{ X_1,\dots,X_{\ell-1}\}$ is calculated. While searching for $X_{\ell}$ the norm of $X_{\ell} - \mathcal{P}_{\ell-1}(X_{\ell})$ is monitored. If it becomes too small, it indicates (since $X_{\ell}$ is normalized) that $X_{\ell}$ is close to linearly dependent on the previously calculated matrices $X_1,\dots,X_{\ell-1}$. The process for $X_{\ell}$ is then randomly restarted in the orthogonal complement of $\spann\{X_1,\dots,X_{\ell-1}\}$ \emph{within} $\mathcal{M}$, which is the range of $\mathcal{Q}_{\ell-1} = \mathcal{P}_{\mathcal{M}} - \mathcal{P}_{\ell-1}$.

\begin{algorithm}[h!]
\linespread{1.2}\selectfont
\DontPrintSemicolon
\KwIn{Orthogonal projection $\mathcal{Q}_{\ell-1}$, matrix $X_{\ell}$ and tolerance $restarttol > 0$.}
\KwOut{Eventually replaced $X_{\ell}$.}
\If{$\| \mathcal{Q}_{\ell-1}(X_{\ell}) \|_F < restarttol$}{
    Replace $X_{\ell}$ by a random element in the range of $\mathcal{Q}_{\ell-1}$.\;
    $X_{\ell} \leftarrow X_{\ell} / \| X_{\ell} \|_F$\;
}
\caption{Restart for linear independence}\label{alg:restart}
\end{algorithm}

In our implementation, we do not apply a random matrix to $\mathcal{Q}_{\ell - 1}$ to obtain a random element in the range. Instead $\mathcal{Q}_{\ell - 1}$ is stored in the form of an orthonormal basis for which random coefficients are computed.
We note that restarting was seldom needed in our experiments, except for the image separation problem in Section~\ref{sec:images}. A qualitative explanation is that the space $\mathcal{Q}_\ell = \mathcal{P}_{\mathcal{M}} - \mathcal{P}_\ell$ from which we (randomly) obtain the next initial guess is rich in the components of matrices that we have not yet computed, thus it is unlikely that an iterate $X_{\ell}$ converges towards an element in $\mathcal{P}_\ell$. 

\subsubsection*{Final algorithm and discussion}

The algorithm we propose for the low-rank basis problem is outlined as Algorithm~\ref{alg:heuristic greedy algorithm}. Some remarks are in order.

\begin{algorithm}[h!]
\linespread{1.2}\selectfont
\DontPrintSemicolon
\KwIn{Basis $M_1,\dots M_d \in \mathbb{R}^{m \times n}$ for $\mathcal{M}$, and integer $restartit > 0$.}
\KwOut{Low-rank basis $X_1,\dots,X_d$ of $\mathcal{M}$.}
Assemble the projection $\mathcal{P}_{\mathcal{M}}$.\;
Set $\mathcal{Q}_0 = \mathcal{P}_{\mathcal{M}}$.\;
\For{$\ell = 1,\dots,d$}{
	Initialize normalized $X_\ell$ randomly in the range of $\mathcal{Q}_{\ell-1}$.\;
    Run Algorithm~\ref{alg:rankest} (Phase I) on $X_\ell$, but every $restartit$ iterations run Algorithm~\ref{alg:restart}.\;
    Run Algorithm~\ref{alg:getmats} (Phase II) on $X_\ell$, but every $restartit$ iterations run Algorithm~\ref{alg:restart}.\;
	Assemble the projection $\mathcal{P}_\ell$ on $\spann\{ X_1,\dots,X_\ell \}$ (ensure linear independence),\;
	Set $\mathcal{Q}_\ell = \mathcal{P}_{\mathcal{M}} - \mathcal{P}_\ell$.\;
}
\caption{Greedy algorithm for computing a low-rank basis}\label{alg:heuristic greedy algorithm}
\end{algorithm}

\begin{enumerate}
\item
The algorithm is not stated very precisely as the choice of many input parameters are not specified. We will specify at the beginning of Section~\ref{sec: experiments} the values we used for numerical experiments.
\item Analogously to~\eqref{eq: projection}, 
 the projections $\mathcal{P}_\ell$ are in practice obtained from a QR decomposition of $[\vec(X_1), \dots, \vec(X_\ell)]$.
\item
Of course, after Algorithm~\ref{alg:rankest} one can or should check whether it is necessary to run Algorithm~\ref{alg:getmats} at all. Recall that Algorithm~\ref{alg:rankest} provides a rank-estimate $r$ and a new matrix $X_\ell \in \mathcal{M}$. There is no need to run Algorithm~\ref{alg:getmats} in case $\rank(X_\ell) = r$ at that point. However, we observed that this seems to happen only when $r=1$ (see next subsection and Section~\ref{sec: synthetic averaged}), so in practice it is enough to check whether $r=1$.
\item
There is a principal difference between the greedy Algorithm~\ref{alg:heuristic greedy algorithm}, and the theoretical Algorithm~\ref{alg:greedy algorithm} regarding the monotonicity of the found ranks. For instance, there is no guarantee that the first matrix $X_1$ found by Algorithm~\ref{alg:heuristic greedy algorithm} will be of lowest possible rank in $\mathcal{M}$, and it will often not happen. 
It seems to depend on the starting value (recall the discussion on the initial guess in Section~\ref{sec:algrank}),  and an explanation may be that the soft thresholding procedure gets attracted by ``some nearest'' low-rank matrix, although a rigorous argument remains an open problem.
In any case, finding a matrix of lowest rank is NP-hard as we have already explained in the introduction. In conclusion, one should hence not be surprised if the algorithm produces linearly independent matrices $X_1,\dots,X_d$ for which the sequence $\rank(X_\ell)$ is not monotonically increasing. Nevertheless, at least in synthetic test, where the exact lowest-rank basis is exactly known and not too badly conditioned, the correct ranks are often recovered, albeit in a wrong order; see Figures~\ref{fig:typicalalg}, \ref{fig:typicalalgn20} and Section~\ref{sec: synthetic averaged}.
\item
It is interesting to note that in case the restart procedure is not activated in any of the loops (that is, the if-clause in Algorithm~\ref{alg:restart} always fails), one would have been able to find the matrices $X_1,\dots,X_d$ independently of each other, e.g., with a random orthogonal basis as a starting guess. In practice, we rarely observed that restart can be omitted, although it might still be a valuable idea to run the processes independently of each other, and monitor the linear independence of $X_1,\dots,X_d$ as a whole. If some of the $X_\ell$ start converging towards the same matrix, or become close to linearly dependent, a modified restart procedure will be required. A practical way for such a simultaneous restart is to use a QR decomposition with pivoting. It will provide a triangular matrix with decreasing diagonal entries, with too small entries indicating basis elements that should be restarted. Elements corresponding to sufficiently large diagonal elements in the QR decomposition can be kept. We initially tried this kind of method, an advantage being that the overall cost for restarts is typically lower. We found that it typically works well when all basis elements have the same rank.  However, the method performed very poorly in situations where the ranks of the low-rank basis significantly differ from each other. Nonetheless, this ``parallel'' strategy might be worth a second look in future work. 
\end{enumerate}

\subsection{Complexity and typical convergence plot}\label{subsec: typical convergence}

The main step in both Algorithms~\ref{alg:rankest} and~\ref{alg:getmats} is the computation of an SVD of an $m \times n$ matrix, which costs about $14mn^2+8n^3$ flops~\cite[Sec.~8.6]{golubbook4th}. This is done at most $2 maxit$ times (assuming the same in both algorithms) for $d$ matrices. 
Since we check the need for restarting only infrequently, 
the cost there is marginal. The overall worst-case complexity of our greedy Algorithm~\ref{alg:heuristic greedy algorithm} hence depends on the choice of $maxit$. In most of our experiments, the inner loops in Algorithms~\ref{alg:rankest} and~\ref{alg:getmats} terminated 
 according to the stopping criteria, long before $maxit$ iterations was reached.

Figure~\ref{fig:typicalalg} illustrates the typical convergence behavior of Algorithm~\ref{alg:heuristic greedy algorithm} for a problem with $m=20$, $n=10$, and $d=5$, 
constructed randomly using the MATLAB function {\tt randn},
and the exact ranks for a basis are $(1,2,3,4,5)$. The colors 
correspond to $\ell=1,\dots,5$: 
these are also indicated at the top of the figure.
For each $\ell$, the evolution of the $n=10$ singular values of $X_\ell$ are plotted during the iteration (always after projection on $\mathcal{M}$). The shaded areas show Phase I, the unshaded areas Phase II. In both phases we used $maxit = 1000$ and $restartit = 50$, moreover, Phase I was terminated when the rank did not change for $changeit = 50$ iterations (which appears to be still conservative), while Phase II was terminated when $X_\ell$ remained unchanged up to a tolerance of $10^{-14}$ in Frobenius norm. The number of SVDs is in principle equal to the number of single iterations, and governs the complexity, so it was used for the x-axis. Illustrating 
the fourth remark above, the ranks are not recovered in increasing order, but in the order $(1,5,3,2,4)$ (corresponding to the number of curves per $\ell$ not converging to zero). 
Repeated experiments suggest that all orderings of rank recovery is possible. 



\begin{figure}[htbp]
 \centering
     \includegraphics[width=\textwidth]{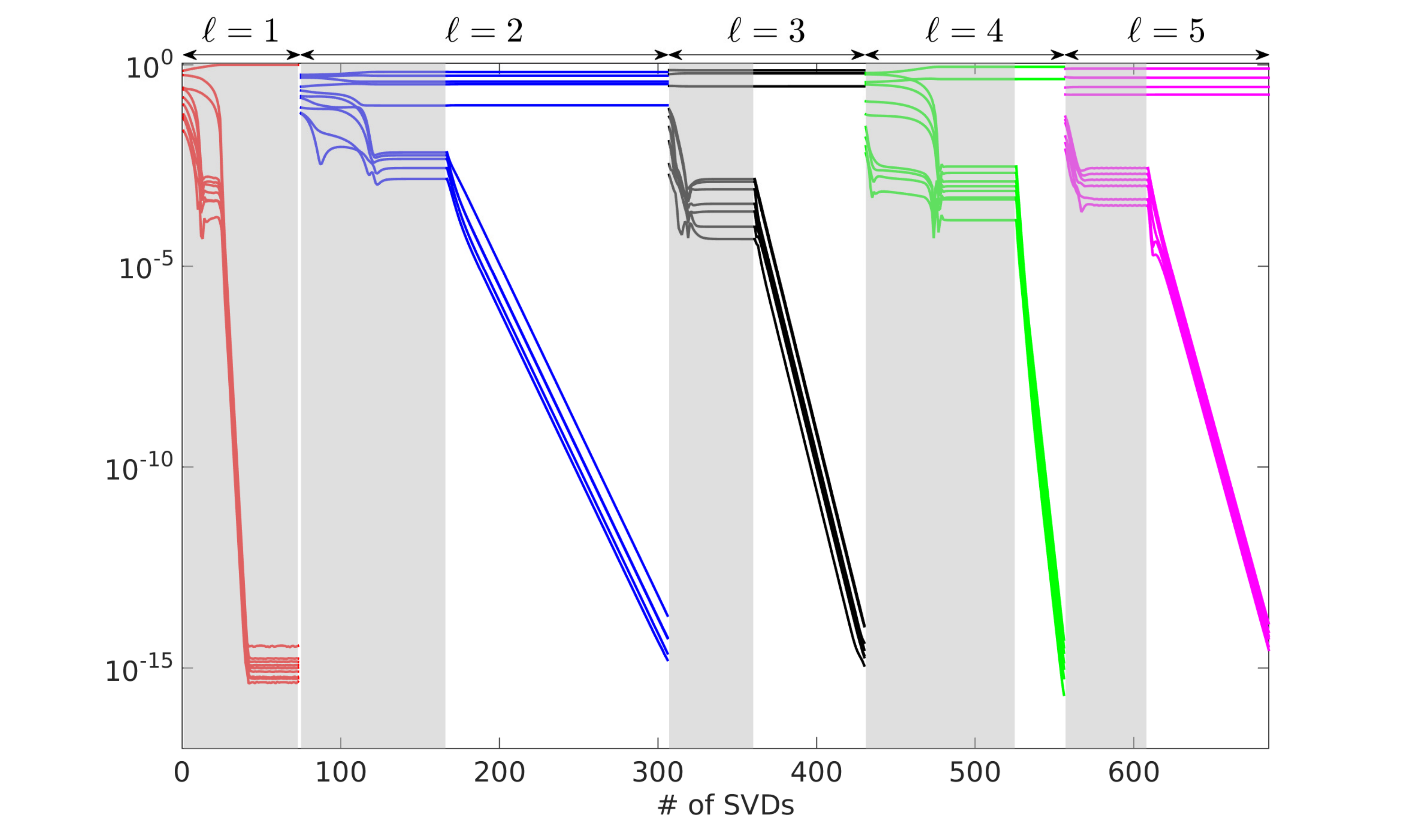}
  \caption{
Typical convergence of Algorithm~\ref{alg:heuristic greedy algorithm} for a randomly generated problem with $m=20, n=10,d=5$. Each color represents one outmost loop $\ell=1,\ldots, 5$. The shaded regions indicate Phase I, and the rest is Phase II. The exact ranks are $(1,2,3,4,5)$, but recovered in the order $(1,5,3,2,4)$.  
}
  \label{fig:typicalalg}
\end{figure}


Regarding the convergence of a single matrix, we make two observations in Figure~\ref{fig:typicalalg}. First, one can nicely see that in Phase I the singular values beyond $\sigma_{r_\ell}$ usually decrease until they stagnate at a certain plateau. The length of this plateau corresponds to the $50$ iterations we have waited until accepting an unchanged rank (before projection) as a correct guess. Except for the first (red) curve, which shows convergence towards a rank-one basis element, the error level of this plateau is too high to count as a low-rank matrix. Put differently, the (normalized) low-rank matrix $Y_\ell$ found by the shrinkage, on which the rank guess is based, is unacceptably far away from the subspace, illustrating the need for Phase II. 
Phase II seems unnecessary only when a rank-one matrix is targeted. 

Second, the convergence of $\sigma_{r+1},\ldots,\sigma_{n}$ towards zero in the second phase is typically linear, and tends to be faster if the limit is lower in rank. We give an 
 explanation for this observation in Section~\ref{sec:convanal}.

The fact that the matrices are obtained in somewhat random order can be problematic 
in some cases, such as the low-rank matrix problem where only one matrix of lowest rank is sought. 
One remedy  is to try multiple initial guesses for Phase I, and adopt the one that 
results in the lowest rank estimate $r$. Figure~\ref{fig:typicalalgrep} is a typical illustration when three 
initial guesses are attempted. The shaded regions represent Phase I repeated three times for each greedy step, and 
the ones that were adopted are shaded darkly. Observe that now the matrices are obtained in the correct rank order $(1,2,3,4,5)$. 

\begin{figure}[htbp]
\centering
     \includegraphics[width=\textwidth]{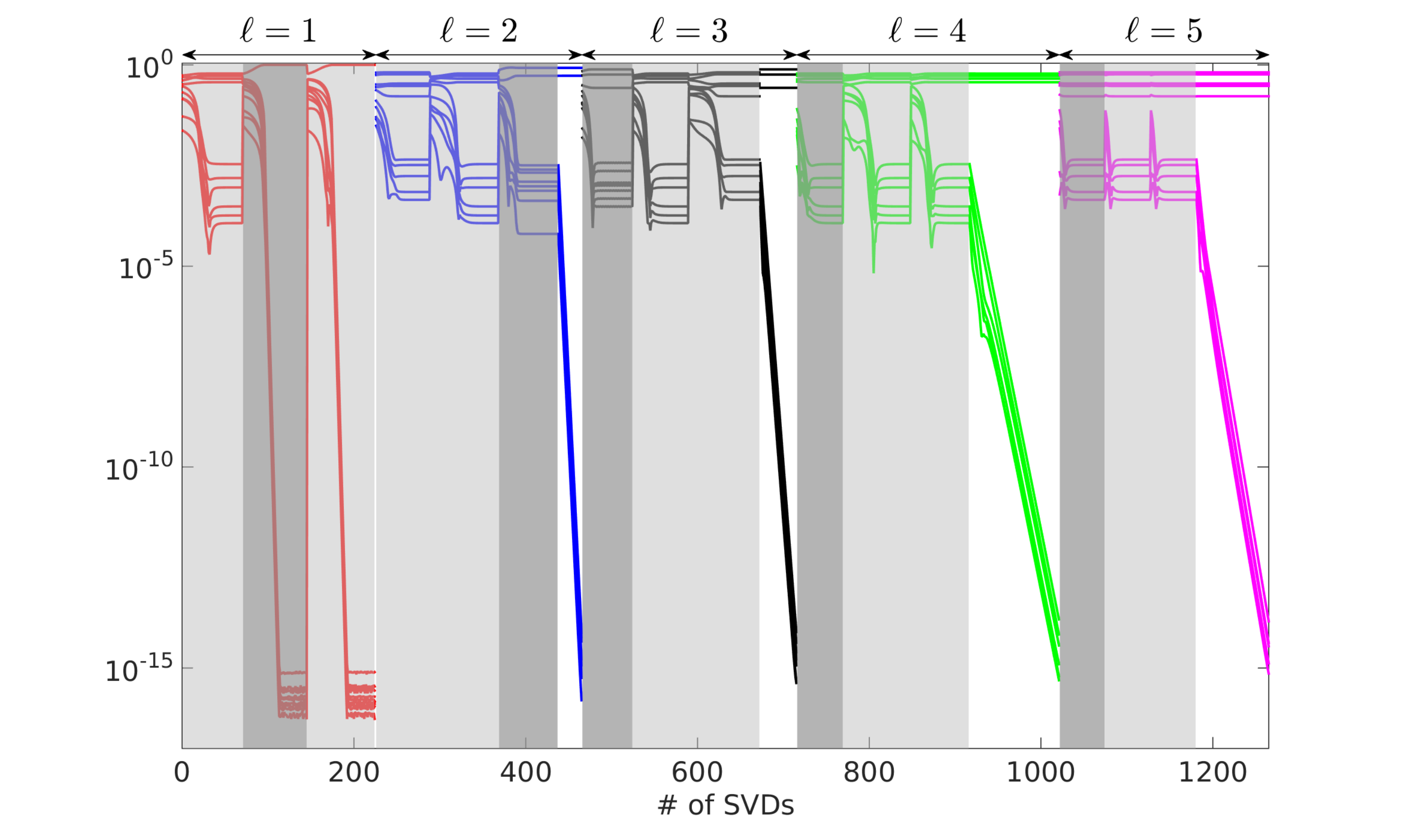}
  \caption{Same settings as Figure~\ref{fig:typicalalg}, but with 
three random initial guesses. Darkly shaded region represent the 
initial guess that was adopted. For instance, for the first matrix (red curve) the rank was estimated to be four in the first run of Phase I, and one in the second and third. The second was adopted, Phase II was not necessary. The final rank order is the correct $(1,2,3,4,5)$.
}
  \label{fig:typicalalgrep}
\end{figure}

Finally, Figure~\ref{fig:typicalalgn20} contains an outcome of the initial experiment (without multiple initial guesses) with square matrices of size $m = n =20$. The ranks are recovered in another ordering, namely 
$(5,1,2,3,4)$.
 One can see that the ratio of the number of iterations in Phases I and II is quite different, and the overall number of required SVDs is smaller. 
The convergence analysis in Section~\ref{sec:convanal}, 
which suggests a typical convergence factor $\sqrt{\frac{r}{n}}$, 
 gives a partial explanation. 
The main reason we give this third plot is the following interesting fact: the maximum rank a matrix in the generated subspace can have is $1+2+3+4+5 = 15$. Consequently, there seems to be a principal difference here from the previous example in that the subspace does not contain a full-rank matrix. 
This is perhaps part of why this problem seems somewhat easier in terms of the total number of iterations. 
Indeed, a close inspection of Figure~\ref{fig:typicalalgn20} reveals that 
for each $\ell$, 
we have five singular values below machine precision (recall that the plot shows the singular values after projection on the subspace).

\begin{figure}[htbp]
\centering
     \includegraphics[width=\textwidth]{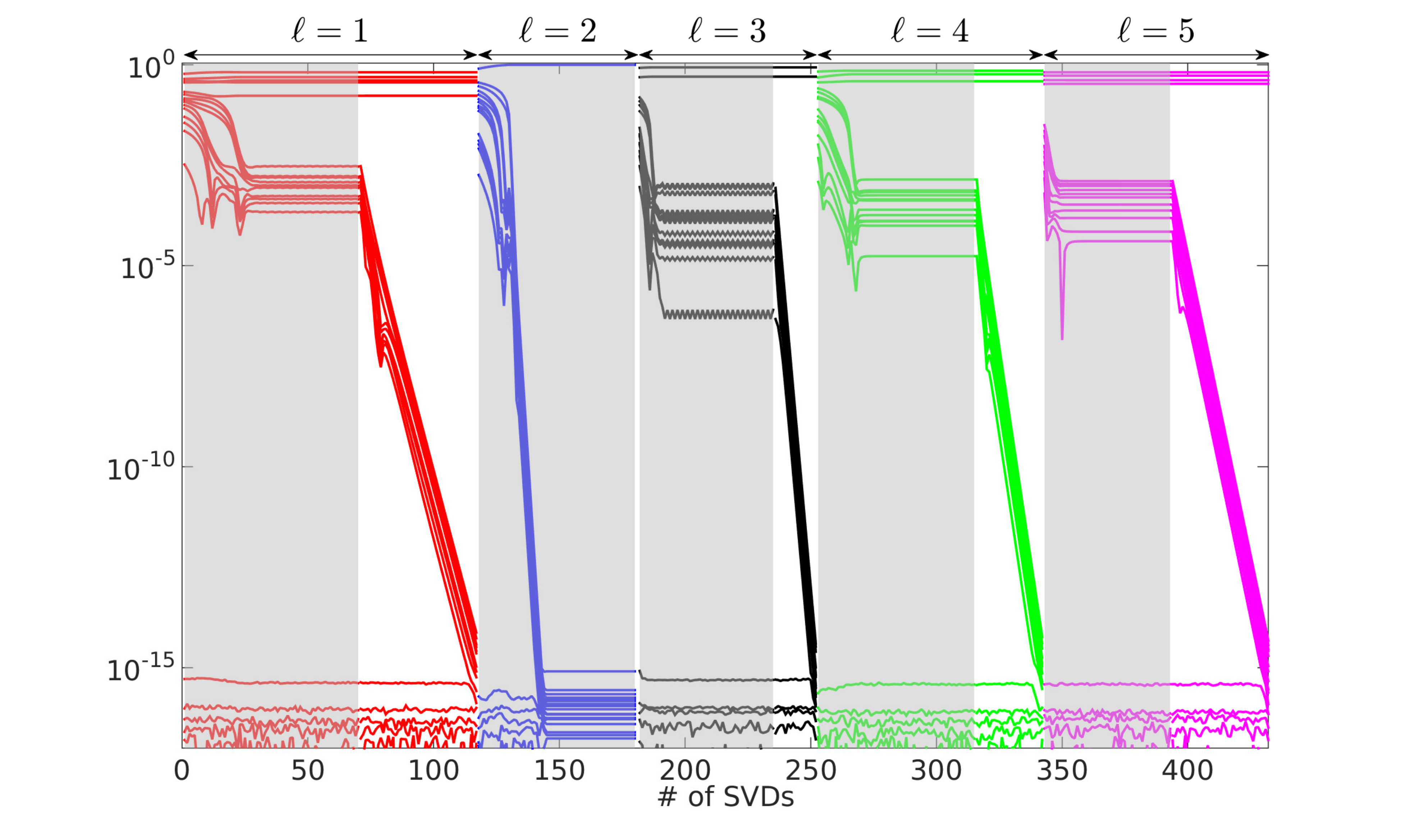}
     \caption{$m=n=20$. 
All matrices in the subspace are rank-deficient, and we observe that the number or SVDs is fewer. 
}     \label{fig:typicalalgn20}
\end{figure}

\section{Convergence analysis}\label{sec:convanal}

Given the hardness of the low-rank basis problem, the formulation of conditions for success or failure of Algorithm~\ref{alg:heuristic greedy algorithm} must be a challenging task. Not to mention the discontinuous nature of the restart procedure, a satisfying rigorous and global convergence result  remains a potentially interesting problem for future work.

%

Here we confine ourselves to a local convergence analysis of Phase II for a single matrix given a correct rank estimate $r$, which is an alternating projection method. We will consider the simplest case where at the point of interest the tangent space of the manifold of rank-$r$ matrices intersects trivially with the tangent space of the sphere in $\mathcal{M}$, which is arguably the simplest possible assumption when it comes to local convergence of the alternating projection method between smooth manifolds.

Regarding Phase I, we can at least note that if $X \in \mathcal{M}$ is a normalized rank-one matrix, then in a neighborhood of $X$ a single step of Algorithms~\ref{alg:rankest} and~\ref{alg:getmats}, respectively, will give the same result (this is also true for some similar choices of $\tau$ discussed above). Under the same assumptions as for Phase II this hence shows the local convergence of Phase I toward isolated normalized rank-one matrices in $\mathcal{M}$, 
see Corollary~\ref{cor: local convergence Phase I} below.
Since 
this is a local convergence analysis, it of course 
 does not fully explain the strong global performance of both Algorithms~\ref{alg:rankest} and~\ref{alg:getmats} in the rank-one case as seen in Figures~\ref{fig:typicalalg}--~\ref{fig:typicalalgn20} and Tables~\ref{tab:rand1} and~\ref{tab:rand5}. 

In general, using the shrinkage operator cannot be guaranteed to converge to a local solution. We have already noted that a matrix of rank larger than two cannot be a fixed point of the shrinkage operator (unless all nonzero singular values are the same). One could construct examples where in a fixed point of~\eqref{eq: update formulas} $X$ has the same rank as $Y$, but generically this seems extremely unlikely. Therefore, the convergence of Phase I is in general a less relevant question. The main problem is in which situations it provides a correct rank estimate, at least locally. 
Except for the rank-one case 
we have no precise arguments, but we give a qualitative explanation at the end of Section~\ref{subsec:convanal}.

\subsection{Local convergence of Phase II}\label{subsec:convanal}

Algorithm~\ref{alg:getmats} is nothing else than the method of alternating projections for finding a matrix in the intersection $\mathcal{B} \cap \mathcal{R}_r$ of the closed sets
\[
\mathcal{B} = \{ X \in \mathcal{M} \vcentcolon \| X \|_F = 1\} 
\]
and
\[
\mathcal{R}_{r} = \{ X \in \mathbb{R}^{m \times n} \vcentcolon \rank(X) \le r \}.
\]

The local convergence analysis of the alternating projection method for nonconvex closed sets has made substantial progress during the last years~\cite{AnderssonCarlsson2013,LewisLukeMalick2009,LewisMalick2008,NollRondepierre2015}, often with reference to problems involving low-rank matrices. In these papers 
one finds very abstract result in the language of variational analysis or differential geometry, but validating the required assumptions for specific situations can be complicated. The most recent work~\cite[Theorem~7.3]{drusvyatskiy2015transversality}, however, contains a very general and comprehensive result that the method of alternating projections is essentially locally convergent (in the sense of distance to the intersection) for two semialgebraic sets of which one is bounded, which is the case here. Moreover, convergence is R-linear and point-wise in case of an \emph{intrinsically transversal} intersection point~\cite[Theorem~6.1]{drusvyatskiy2015transversality}. A notable special case, for which the result has already been obtained in~\cite[Theorem~5.1]{AnderssonCarlsson2013}, is a \emph{clean} (or \emph{nontangential}) intersection point. As the assumption~\eqref{eq: nontangential intersection} made below implies such a clean intersection for the problem at hand, the subsequent Theorem~\ref{th: local convergence} follows. Nevertheless, we provide a direct and self-contained proof for the problem at hand based on elementary linear algebra, which may contribute to the understanding of alternating projection method when specifically used for low-rank constraints. 
In Remark~\ref{remark on condition} we discuss alternatives to the main assumption~\eqref{eq: nontangential intersection} in the context of the available literature in a bit more detail, and provide a sufficient condition for~\eqref{eq: nontangential intersection} in Lemma~\ref{lem: suff. condition for local convergence}.


To state the result, we assume that $X_* \in \mathcal{T}_r \cap \mathcal{B}$ has exactly rank $r$. By semi-continuity of rank, 
 all matrices in a neighborhood (in $\mathbb{R}^{m \times n}$) of $X$ have rank at least $r$. Therefore, we can locally regard Algorithm~\ref{alg:getmats} as an alternating projection between $\mathcal{B}$ and the smooth manifold of matrices of fixed rank $r$. Letting $X_* = U_*^{} \Sigma_*^{} V^T_*$ be a thin SVD of $X_*$ where $\Sigma_*$ contains positive singular values, the tangent space to that manifold at $X$ is~\cite{HelmkeShayman1995}
\begin{equation}\label{eq: tangent space of fixed rank}
T_{X_*} \mathcal{R}_r = \left\{[U_*^{}\ U_*^\perp] \begin{bmatrix} A & B \\ C & 0 \end{bmatrix} [V_*^{}\ V_*^\perp] \vcentcolon A \in \mathbb{R}^{r \times r}, \ B \in \mathbb{R}^{r \times (m-r)},\ C \in \mathbb{R}^{(n-r) \times r} \right\}.
\end{equation}
For our analysis we make the following genericity assumption:
\begin{equation}\label{eq: nontangential intersection}
T_{X_*} \mathcal{R}_r \cap T_{X_*} \mathcal{B} = \{ 0 \}.
\end{equation}
Here $T_{X_*} \mathcal{B}$ is the tangent space of $\mathcal{B}$, which is the orthogonal complement of $X_*$ within $\mathcal{M}$. Since $T_{X_*} \mathcal{R}_r$ contains $X_*$, ~\eqref{eq: nontangential intersection} is expressed in terms of $\mathcal{M}$ as
\begin{equation}\label{eq: equivalent non-tangential condition}
T_{X_*} \mathcal{R}_r \cap \mathcal{M} = \spann\{X_*\}.
\end{equation}
We remark that~\eqref{eq: nontangential intersection} ultimately implies that $X_*$ is an isolated point of $\mathcal{R}_r \cap \mathcal{B}$, so actually it then holds $\mathcal{R}_r \cap \mathcal{M} = \spann\{X_*\}$. 
Then we have the following result.

\begin{theorem}\label{th: local convergence}
Assume $X_* \in \mathcal{T}_r \cap \mathcal{B}$ has rank $r$, and that this rank is used in Algorithm~\ref{alg:getmats}. If~\eqref{eq: nontangential intersection} holds, then for $X \in \mathcal{B}$ close enough to $X_*$ 
the new iterate $X_{\text{new}}=\frac{ \mathcal{P}_{\mathcal{M}}(\mathcal{T}_r(X))}
{\| \mathcal{P}_{\mathcal{M}}(\mathcal{T}_r(X))\|_F}$ constructed by the algorithm is uniquely defined, and 
\begin{equation*}
  \label{eq:thmstate2}
 \frac{
\left\|X_{new}
 - X_* \right\|_F}{\| X - X_*\|_F} \le 
 \cos \theta + O(\| X - X_*\|_F^2),
\end{equation*}
where $\theta \in (0,\frac{\pi}{2}]$ is the subspace angle between $T_{X_*} \mathcal{B}$ and $T_{X_*} \mathcal{R}_r$, defined by 
\begin{equation*}  \label{eq:defangle}
\cos\theta = \max_{\substack{ X \in T_{X_*}\mathcal{B} \\ Y\in T_{X_*}\mathcal{R}_r} } \frac{\abs{\langle X, Y \rangle_F}}{\|X\|_F\|Y\|_F}. 
\end{equation*}
As a consequence, the iterates produced by Algorithm~\ref{alg:getmats} are uniquely defined and converge to $X_*$ at a linear rate for close enough starting values.
\end{theorem}

In accordance with our observations in Section~\ref{subsec: typical convergence}, we also have a result for Phase I in the rank-one case.

\begin{corollary}\label{cor: local convergence Phase I}
If $r=1$, the statement of Theorem~\ref{th: local convergence} also holds for Algorithm~\ref{alg:rankest}. In fact, both algorithms produce the same iterates in some neighborhood of $X_*$.
\end{corollary}

Here it is essential that the value of shift $\tau$ is bounded below by a fixed fraction of the Frobenius norm of $X$, as it is the case for in Algorithm~\ref{alg:rankest}, a lower bound being $\delta / \sqrt{\min(m,n)}$.

\begin{theopargself}
\begin{proof}[of Theorem~\ref{th: local convergence}]
Since $X$ and $X_*$ are in $\mathcal{B}$, we can write
\[
X - X_* = E + O(\| X - X_* \|_F^2)
\]
with $E \in T_{X_*}\mathcal{B}$. We partition the principal error $E$ as
\begin{equation}  \label{eq:EE}
E = [U_*^{}\ U_*^\perp] \begin{bmatrix} A & B \\ C & D \end{bmatrix} [V_*^{}\ V_*^\perp]^T.  
\end{equation}
Of course, $\| E \|_F^2 = \|A\|_F^2 + \| B \|_F^2 + \|C\|_F^2 + \|D\|_F^2$, and due to~\eqref{eq: tangent space of fixed rank}, our assumption implies
\begin{equation}\label{eq: fraction}
\| D \|_F^2 \ge 
\sin^2 \theta \cdot \| E \|_F^2.
\end{equation}

Since $X_*$ has rank $r$, it follows that all matrices in some neighborhood have a unique best rank-$r$ approximation (by perturbation arguments for the singular values). In this neighborhood $X_{\text{new}}$ is uniquely defined. To relate $\|E\|$ to $\|\mathcal{T}_r(X) - X_*\|$
%
%
%
%
we consider the two matrices
\[
F = 
    \begin{bmatrix}
      I&  C\Sigma_*^{-1} \\
      -C\Sigma_*^{-1}& I
    \end{bmatrix}[U_*^{}\ U_*^\perp]^T,
\]
and
\[
G =
     [V_*^{}\ V_*^\perp]\begin{bmatrix}
      I& -\Sigma_*^{-1}B\\
      \Sigma_*^{-1}B & I
    \end{bmatrix},
\]
both of which are orthogonal up to $O(\|E\|_F^2)$:
\[
\|F^TF-I\|_F =O(\|E\|_F^2), \quad \|G^TG-I\|_F=O(\|E\|_F^2).\footnote{Here and in the following, the error constant behind $O(\|E\|_F)$ depends mainly on the condition of $\Sigma_*$, which can be very large, but is fixed in this type of local analysis.}
\]
Therefore, denoting by $\tilde F,\tilde G$ the orthogonal polar factors of $F,G$, respectively, we also have\footnote{From a polar decomposition $Z^T=UP$ one gets $Z^TZ-I = (Z^T-U)(P+I)U^T$, and since the singular
  values of $(P+I)U^T$ are all at least 1, it follows that $\|Z-U^T\|_F\leq
  \|Z^TZ-I\|_F$.}
\begin{equation}\label{eq: relation to polar factor}
 F = \tilde{F} + O(\|E\|^2), \quad G = \tilde{G} + O(\|E\|_F^2).
\end{equation}
One now verifies that
\[
  \tilde F X \tilde G = F X G +O(\| E \|_F^2) = 
  \begin{bmatrix}
    \Sigma_*+A & 0\\0&D
  \end{bmatrix}+O(\| E \|_F^2),
\]
or, since $\tilde{F}$ and $\tilde{G}$ are orthogonal,
\[
 X = \tilde{F}^T \begin{bmatrix}
    \Sigma_*+A & 0\\0&D
    \end{bmatrix}
    \tilde{G}^T
 + O(\|E\|_F^2).
\]
For $E$ small enough, the best rank-$r$ approximation of the principal part is obtained by deleting $D$. Hence, 
\begin{align}
\mathcal{T}_r(X) &=  \mathcal{T}_r\left( \tilde{F}^T\begin{bmatrix}
    \Sigma_*+A & 0\\0&D
    \end{bmatrix}\tilde{G}^T \right) + O(\| E\|_F^2) \nonumber\\
    &= \tilde{F}^T\begin{bmatrix}
    \Sigma_*+A & 0\\0& 0
    \end{bmatrix}\tilde{G}^T + O(\| E\|_F^2).\label{eq:trx}
\end{align}
To get the last equality we have used 
results from  matrix perturbation theory~\cite{wedin72} (see also \cite[Sec.V.4]{stewart-sun:1990}), which shows that under the perturbation $O(\|E\|_F^2)$, 
the singular subspace corresponding to the $r$ largest singular values 
of $X$ gets perturbed by $O(\frac{\|E\|_2^2}{gap})$ where $gap$ is the smallest distance between the singular values of $\Sigma_*+A$ and those of $D$. 
For $\|E\|_F$ sufficiently small such that $\|E\|_F=o(\sigma_{\min}(\Sigma_*))$, 
this bound is $O(\|E\|_2^2)$. 
Together with the fact that the perturbation in the singular values is bounded 
also by $O(\|E\|_2^2)$ (since the condition number of singular values is always 1), we obtain the final equality above. 

Therefore, taking also~\eqref{eq: relation to polar factor} into account, 
we obtain 
\begin{align*}
\| \mathcal{T}_r(X) - X_* \|_F  &= \|\tilde{F} \mathcal{T}_r(X) \tilde{G} - F X_* G \|_F + O(\|E\|_F^2) \\
&= \left\| \begin{bmatrix} A + \Sigma_* & 0 \\ 0 & 0 \end{bmatrix} - \begin{bmatrix} \Sigma_* & - B \\ -C & C\Sigma_*^{-1}B \end{bmatrix} \right\|_F + O(\|E\|_F^2)\\
&= \left\| \begin{bmatrix} A & B \\ C & 0 \end{bmatrix} \right\|_F + O(\|E\|_F^2).
\end{align*}
Here we used $O(\|C\Sigma_*^{-1}B\|_F^2) = O(\|E\|_F^2)$, which holds 
since $\|\Sigma_*^{-1}\|_2$ can be regarded as a constant that does not depend on $\|E\|_F$. 
Since $O(\|E\|) = O(\| X - X_* \|)$, we arrive at
\begin{equation}\label{eq: almost final estimate}
\begin{aligned}
\frac{\| \mathcal{T}_r(X) - X_* \|_F}{\| X - X_*\|_F} &= \frac{\sqrt{\| E \|_F^2 - \| D \|^2_F} + O(\|X - X_* \|_F^2)}{\| E \|_F + O(\| X - X_* \|_F^2)} \\ &\le \sqrt{1 - \sin^2\theta} + O(\| X - X_* \|_F^2)\\
 &= \cos\theta + O(\| X - X_* \|_F^2),
\end{aligned}
\end{equation}
where we have used~\eqref{eq: fraction}. 
Since $X_* \in \mathcal{M}$, it now follows that
\begin{equation}\label{eq: estimate after projection}
\| P_{\mathcal{M}}(\mathcal{T}_r(X)) - X_* \|_F \le \| \mathcal{T}_r(X) - X_* \|_F \le \cos\theta \| X - X_*\|_F + O(\| X - X_* \|_F^3).
 \end{equation}
Finally, we consider the normalization step.
Recalling $\|X_*\|_F=1$, 
by a simple geometric argument on the unit sphere 
we obtain 
\begin{equation}\label{eq: normalization perturbation bound}
\left\| \frac{Y}{ \|Y\|_F} - X_* \right\|_F \leq \frac{1}{\cos\phi}\|Y - X_*\|_F,
\end{equation}
where $\phi\in[0,\frac{\pi}{2}]$ such that $\sin\phi = \|Y-X_*\|_F$. By Taylor expansion,  $\frac{1}{\sqrt{1-\xi^2}} =  1 + O(\xi^2)$. Substituting $\xi = \sin \phi$ and 
$Y = \mathcal{P}_{\mathcal{M}}(\mathcal{T}_r(X))$ in~\eqref{eq: normalization perturbation bound} gives 
\[
\| X_{\text{new}} - X_* \|_F \le \| P_{\mathcal{M}}(\mathcal{T}_r(X)) - X_* \|_F + O(\| P_{\mathcal{M}}(\mathcal{T}_r(X)) - X_* \|_F^3).
\]
Using~\eqref{eq: estimate after projection}, we arrive at
\[
\| X_{\text{new}} - X_* \|_F \le \cos \theta \| X - X_* \|_F + O(\|X - X_*\|_F^3),
\]
completing the proof.
\qed\end{proof}
\end{theopargself}

From Theorem~\ref{th: local convergence} we can obtain 
a rough estimate for the convergence factor that we can expect to observe in practice. 
Consider the ``generic'' case where 
the error term $E$ in~\eqref{eq:EE} 
is randomly distributed, that is, each element is of comparable absolute value. 
Then we have $\|D\|_F^2 \approx \frac{(n-r)^2}{n^2}\|E\|_F^2$, and plugging this into~\eqref{eq: almost final estimate} gives 
$\frac{\| \mathcal{T}_r(X) - X_* \|_F}{\| X - X_*\|_F}\leq \frac{\sqrt{2nr+r^2}}{n}+O(\| X - X_* \|_F^2)$. 
This suggests that we typically expect a convergence factor $\approx O(\sqrt{\frac{r}{n}})$. This estimate reflects the experiments quite well; see Section~\ref{sec:exconv}.

The above proof provides some insight into the behavior of 
Algorithm~\ref{alg:rankest} in Phase I. In this case $\mathcal{T}_r(X)$ in~\eqref{eq:trx} is replaced by $\mathcal{S}_\tau(X)$. Provided again that we start with a matrix $X$ close to $X_*$ so that $\|D\|_2\leq \tau$, the operation $\mathcal{S}_\tau(X)$ again removes the $D$ term, emphasizing the components towards $X_*$ just like in Phase II as shown above. However, now the $\Sigma_*+A$ term is also affected, and thus Phase I stagnates where the thresholding effect in $\Sigma_*+A$ is balanced with the error terms that come in from the projection $\mathcal{P}_{\mathcal{M}}$. Then the rank estimate $r$ is of correct rank $\rank(X_*)$, but neither $X$ nor $Y$ in Algorithm~\ref{alg:rankest} is close to $X_*$; reflecting the remark at the end of Section~\ref{sec:algrank}. 

\begin{remark}\label{remark on condition}
The conditions~\eqref{eq: nontangential intersection}, resp.~\eqref{eq: equivalent non-tangential condition}, allow for a 
 simple proof but of course impose some restrictions, most obviously $d = \dim (\mathcal{M}) \le (m-r)(n-r) + 1$. In a seminal attempt, Lewis and Malick~\cite{LewisMalick2008} obtained the local convergence of the method of alternating projections between two smooth submanifolds $\mathcal{M}$ and $\mathcal{N}$ of $\mathbb{R}^n$ towards some $X_* \in \mathcal{M} \cap \mathcal{N}$ under the condition that $T_{X_*} \mathcal{M} + T_{X_*} \mathcal{N} = \mathbb{R}^n$ (transversality). This allows for a non-trivial intersection, but imposes lower bounds on the dimensions, in our case $d \ge (m-r)(n-r) + 1$. Andersson and Carlsson~\cite{AnderssonCarlsson2013} relaxed the condition to $T_{X_*}( \mathcal{M} \cap \mathcal{N}) = T_{X_*} \mathcal{M} \cap T_{X_*}\mathcal{N}$ under the assumption that $\mathcal{M} \cap \mathcal{N}$ is a $C^2$ manifold. This does not impose restriction on the dimensions, and contains~\eqref{eq: equivalent non-tangential condition} as a special cases. Still these conditions can fail in the situation at hand ($\mathcal{M}$ a subspace of $\mathcal{R}^{m \times n}$, $\mathcal{N} = \mathcal{R}_r$), for instance when $\mathcal{M} = T_{X_*}\mathcal{R}_r$, to mention one counter-example. As already mentioned, the recent results of Drusvyatskiy, Ioffe and Lewis~\cite{drusvyatskiy2015transversality} subsumes the previous results under the more general condition of intrinsic transversality. Another recent work, by Noll and Rondepierre \cite{NollRondepierre2015}, contains local convergence for alternating projections under very weak but abstract assumptions, which do not involve tangential conditions in first place. It may be that their result applies to our setting, but we have not been able to validate this.
\end{remark}

We conclude with a sufficient condition for~\eqref{eq: nontangential intersection} which might be useful in some very structured cases; see Proposition~\ref{prop: relation of full rank factor matrices to tangential interestcion} in the appendix for an example. We denote by $\ran (X)$ and $\ran(X^T)$ the column and row space of a matrix $X$, respectively.

\begin{lemma}\label{lem: suff. condition for local convergence}
Let $X_* \in \mathcal{M}$ have rank $r$. Assume there exists a $(d-1)$-dimensional subspace $\tilde{\mathcal{M}} \subseteq \mathcal{M}$ complementary to $\spann\{ X_*\}$ with the following property: for every $\tilde{X} \in \tilde{\mathcal{M}}$ it holds $
\ran(X_*)\cap \ran(\tilde X)=0$ and $\ran(X_*^T)\cap \ran(\tilde X^T)=0$. 
Then~\eqref{eq: nontangential intersection} holds.
\end{lemma}

\begin{proof}
Consider $X = \alpha X_* + \beta \tilde{X} \in \mathcal{M}$ with $\tilde{X} \in \tilde{\mathcal{M}}$. Let $P$ and $Q$ denote the orthogonal projections on $\ran(X_*)^\bot$ and $\ran(X_*^T)^\bot$, respectively. Then $X \in T_{X_*} \mathcal{R}_r$ if and only if
\[
0 = P^{} X Q^T = \beta P^{} \tilde{X} Q^T.
\]
It holds $P^{} \tilde{X} Q^T \neq 0$. To see this we note that $P^{} \tilde{X} Q^T = 0$ would mean $\ran(\tilde{X} Q^T) \subseteq \ran(X_*)$, which by assumption implies $\tilde{X} Q^T = 0$. But then $\ran(\tilde{X}^T) \subseteq \ran(X_*^T)$, a contradiction. Hence $X \in T_{X_*}\mathcal{R}_r$ if and only if $\beta = 0$, which proves the equivalent condition~\eqref{eq: equivalent non-tangential condition}.
\qed\end{proof}

\section{Experiments}\label{sec: experiments}

Unless stated otherwise, the standard parameters in the subsequent experiments were $\tau_{tol} = 10^{-3}$, $\delta = 0.1$ and $changeit = 50$ in Alg~\ref{alg:rankest}, and
$maxit = 1000$ and $restartit = 50$ for both Phase I and Phase II in Algorithm~\ref{alg:heuristic greedy algorithm}. The typical termination tolerance in Phase II was $tol = 10^{-14}$. The restart tolerance $restarttol$ in Alg.~\ref{alg:restart} was set to $10^{-3}$ but was almost never activated. 

\subsection{Synthetic averaged examples}\label{sec: synthetic averaged}

We fix $m = n = 20$, $d=5$, and a set of ranks $(r_1,\dots,r_d)$. We then randomly generate $d$ random matrices $M_1,\dots,M_d$ of corresponding rank by forming random matrices $U_\ell \in \mathbb{R}^{m \times r_\ell}$ and $V_\ell \in \mathbb{R}^{n \times r_\ell}$ 
with orthonormal columns, obtained from the QR factorization of random Gaussian matrices using MATLAB's ${\tt randn}$
and setting $M_\ell = U_\ell^{} V_\ell^T$. To check the average rate of success of Algorithm~\ref{alg:heuristic greedy algorithm}, we run it 100 times and calculate
\begin{itemize}
\item
the average sum of ranks $\sum_{\ell = 1}^d \rank(Y_\ell)$ found by Phase I of the algorithm,
\item
the average truncation error $\left(\sum_{\ell = 1}^d \| X_\ell - \mathcal{T}_{r_\ell}(X_\ell) \|_F^2\right)^{1/2}$ after Phase I,
\item
the average truncation error $\left(\sum_{\ell = 1}^d \| X_\ell - \mathcal{T}_{r_\ell}(X_\ell) \|_F^2\right)^{1/2}$ after Phase II, 
\item the average iteration count ($\#$ of SVDs computed) in each Phase. 
\end{itemize}
Table~\ref{tab:rand1} shows the results for some specific choices of ranks. 
Phase II  was always terminated using $tol = 10^{-14}$, and never took  the maximum $1000$ iterations. 
From Table~\ref{tab:rand1} we see that the ranks are sometimes estimated incorrectly, although this does not necessarily tarnish the final outcome. 

{\renewcommand{\arraystretch}{1.2}
\begin{table}[htbp]
\begin{center}
\caption{Synthetic results, random initial guess.}\label{tab:rand1}
\begin{tabular}{c|c|c|c}
\text{exact ranks} & \text{av. sum(ranks)} & \text{av. Phase I err (iter)} & \text{av. Phase II err  (iter)}\\
\hline
(  1 ,  1 ,  1 ,  1 ,  1 )&    5.05 & 2.59e-14 (55.7) & 7.03e-15  (0.4) \\
(  2 ,  2 ,  2 ,  2 ,  2 )&   10.02 & 4.04e-03 (58.4) & 1.04e-14 (9.11) \\
(  1 ,  2 ,  3 ,  4 ,  5 )&   15.05 & 6.20e-03 (60.3) & 1.38e-14 (15.8) \\
(  5 ,  5 ,  5 , 10 , 10 )&   35.42 & 1.27e-02 (64.9) & 9.37e-14 (50.1) \\
(  5 ,  5 , 10 , 10 , 15 )&   44.59 & 2.14e-02 (66.6) & 3.96e-05 (107) \\
\end{tabular}
\end{center}
\end{table}

A simple way to improve the rank estimate is to repeat Phase I with several initial matrices, and adopt the one that results in the smallest rank. Table~\ref{tab:rand5} shows the results obtained in this way using five random initial guesses. 

\begin{table}[htbp]
\begin{center}
\caption{Synthetic results, random initial guess from subspace repeated 5 times. 
}\label{tab:rand5}
\begin{tabular}{c|c|c|c}
\text{exact ranks} & \text{av. sum(ranks)} & \text{av. Phase I err (iter)} & \text{av. Phase II err  (iter)}\\
\hline
(  1 ,  1 ,  1 ,  1 ,  1 )&    5.00 & 6.77e-15  (709) & 6.75e-15 (0.4)\\
(  2 ,  2 ,  2 ,  2 ,  2 )&   10.00 & 4.04e-03  (393) & 9.57e-15 (9.0)\\
(  1 ,  2 ,  3 ,  4 ,  5 )&   15.00 & 5.82e-03  (390) & 1.37e-14 (18.5) \\
(  5 ,  5 ,  5 , 10 , 10 )&   35.00 & 1.23e-02  (550) & 3.07e-14 (55.8) \\
(  5 ,  5 , 10 , 10 , 15 )&   44.20 & 2.06e-02  (829) & 8.96e-06 (227)\\
\end{tabular}
\end{center}
\end{table}
}



We observe that the problem becomes more difficult when the ranks vary widely. 
As mentioned in Section~\ref{sec:algrank}, 
choosing the initial guesses as in~\cite{QuSunWright2014} 
also worked fine, but not evidently better than random initial guesses as in Table~\ref{tab:rand5}.
From the first rows in both tables we validate once again that for the rank-one case, Phase II is not really necessary -- Phase I is recovering a rank-one basis reliably.

\subsubsection{Comparison with tensor CP algorithm}\label{sec:comparetensor}



As we describe in Appendix~\ref{sec: reuction to tensors}, if the 
subspace is spanned by rank-one matrices, then the CP decomposition (if successfully computed; the rank is a required input) of a tensor with slices $M_k$, where $M_1,\dots,M_d$ is any basis of $\mathcal{M}$, provides a desired rank-one basis. Here we compare our algorithm with the CP-based approach. Specifically, we compare with the method \texttt{cpd} in Tensorlab~\cite{sorbertensorlab} with the exact decomposition rank (the dimension $d$ of $\mathcal{M}$) as input. 
By default, this method is based on alternating least-squares with initial guess obtained by an attempt of simultaneous diagonalization. 
When applied to a rank-one basis problem, \texttt{cpd} often gives an accurate CP decomposition with no ALS iteration.

As seen from the tables above, given a rank-one basis problem, our Algorithm~\ref{alg:heuristic greedy algorithm} will typically terminate after Phase I. On the other hand, since we assume a rank-one basis to exist (otherwise the CP approach is not necessarily meaningful for finding a subspace basis), we can also use the alternating projection algorithm from Phase II with rank one directly from random initializations. 
In summary, we obtain three error curves: one for tensorlab, one for soft thresholding (Phase I) and one for alternating projection (Phase II). 
The errors are computed as in the experiments in Section~\ref{sec:algconv} via the subspace angle. 

 We also present the runtime to show that our algorithms are not hopelessly slow in the special rank-one case. 
Just running Phase II results in an algorithm faster than Algorithm~\ref{alg:heuristic greedy algorithm}, but it is still slower than \texttt{cpd}.
Note that while Tensorlab is a highly tuned toolbox, we did not try too hard to optimize our code regarding the choice of parameters and memory consumption. More importantly, unlike \texttt{cpd} our algorithm does not require the rank $r$ and is applicable even when $r>1$.


\paragraph{Growing matrix size $n$}\label{sec:matsize}
We first vary the matrix size $n$, fixing the other parameters. 
The runtime and accuracy are shown in Figure~\ref{fig:tensorex1}. 
We observe that if the CP rank is known, the CP-based algorithm is both fast and accurate. 



\begin{figure}[htbp]
\begin{minipage}{.499\textwidth}
\centering
     \includegraphics[width=66mm]{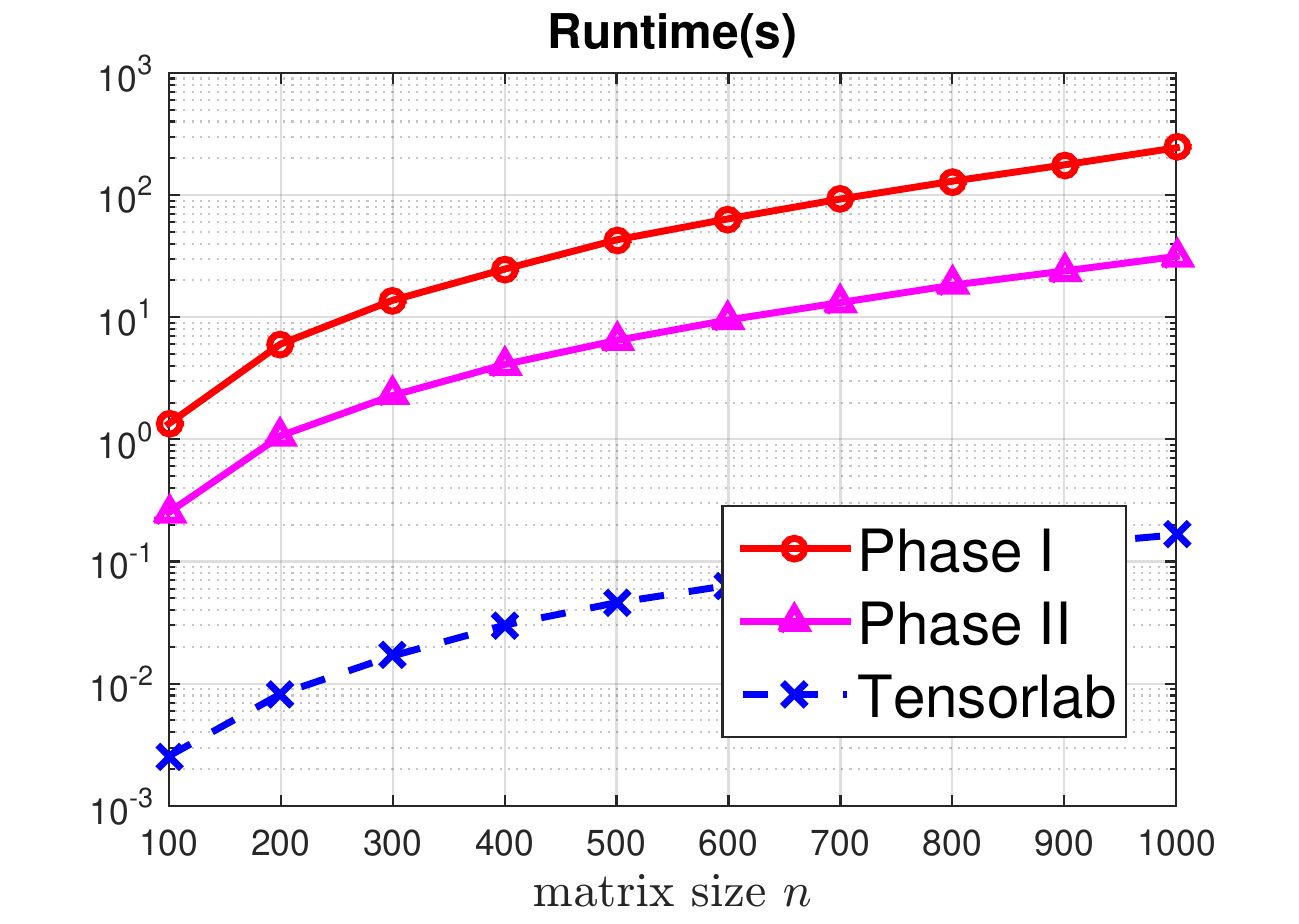}
\end{minipage}
\begin{minipage}{.499\textwidth}
\centering
    \includegraphics[width=66mm]{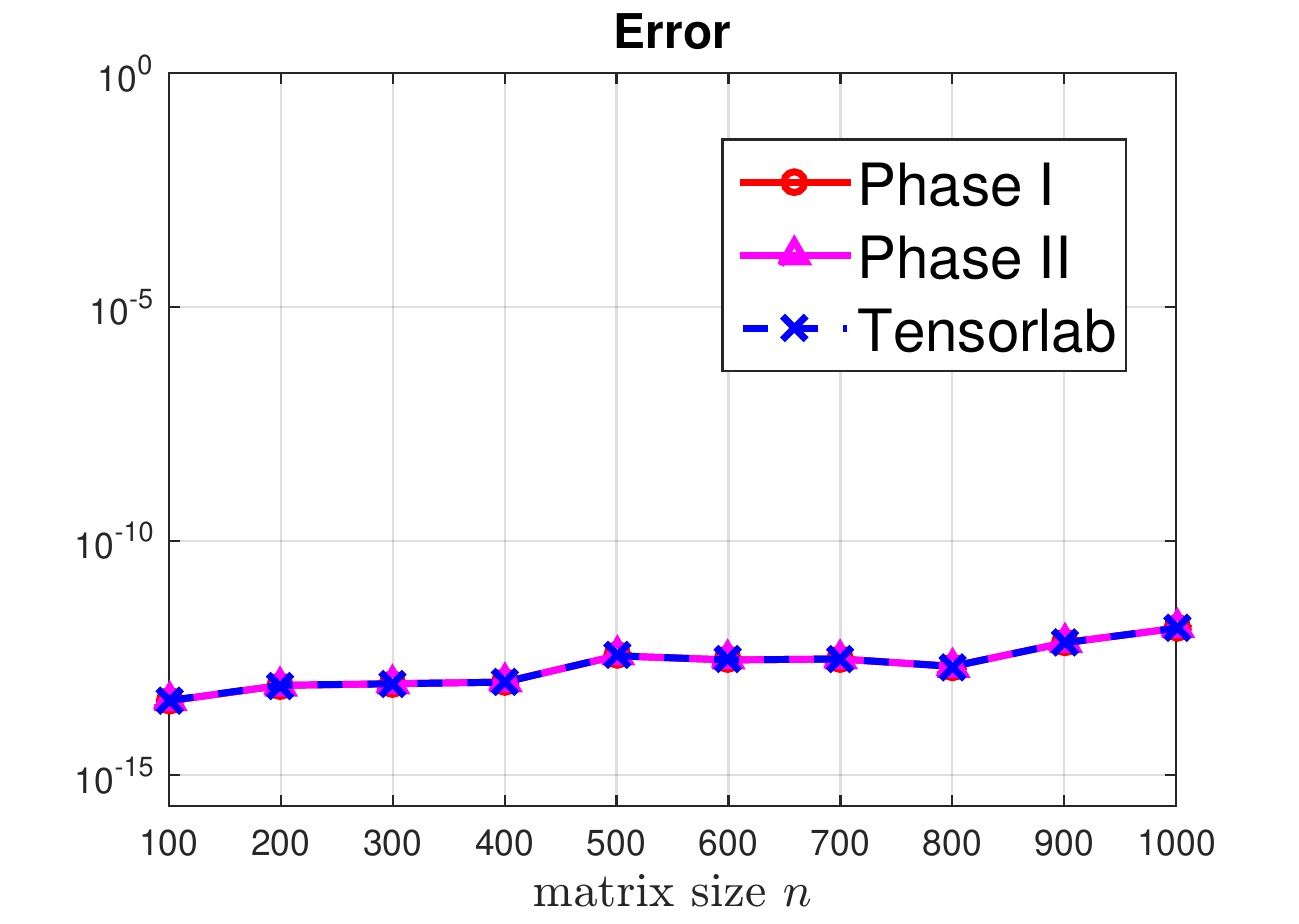}
\end{minipage}
  \caption{
Rank-1 basis matrices $r=1$, 
fixed $d=10$, varying $m=n$ between $50$ and $500$. 
The accuracy is not identical but nearly the same. 
Tensorlab performs well. 
}
  \label{fig:tensorex1}
\end{figure}

\paragraph{Growing dimension $d$}\label{sec:dimsize}
We next vary the dimension $d$, in particular allowing it to 
exceed $n$ (but not $n^2$). In this case, linear dependencies among the left factors $\mathbf{a}_\ell$ and right factors $\mathbf{b}_\ell$, respectively, of a rank-one basis $\mathbf{a}_1^{} \mathbf{b}_1^T, \dots, \mathbf{a}_d^{} \mathbf{b}_d^T$ must necessarily occur. It is known that in this scenario obtaining an exact CP decomposition via simultaneous diagonalization, as in part attempted by \texttt{cpd}, becomes a much more subtle problem, see the references given in Section~\ref{sec: simultaneous diagonalization}. And indeed, we observe that for $d>n$ the accuracy of Tensorlab deteriorates, while our methods do not. The runtime and accuracy for $n=10$ are shown in Figure~\ref{fig:tensorex2}. 
However, 
 this effect was less pronounced for larger $m=n$.


\begin{figure}[htbp]
\begin{minipage}{.499\textwidth}
\centering
     \includegraphics[width=66mm]{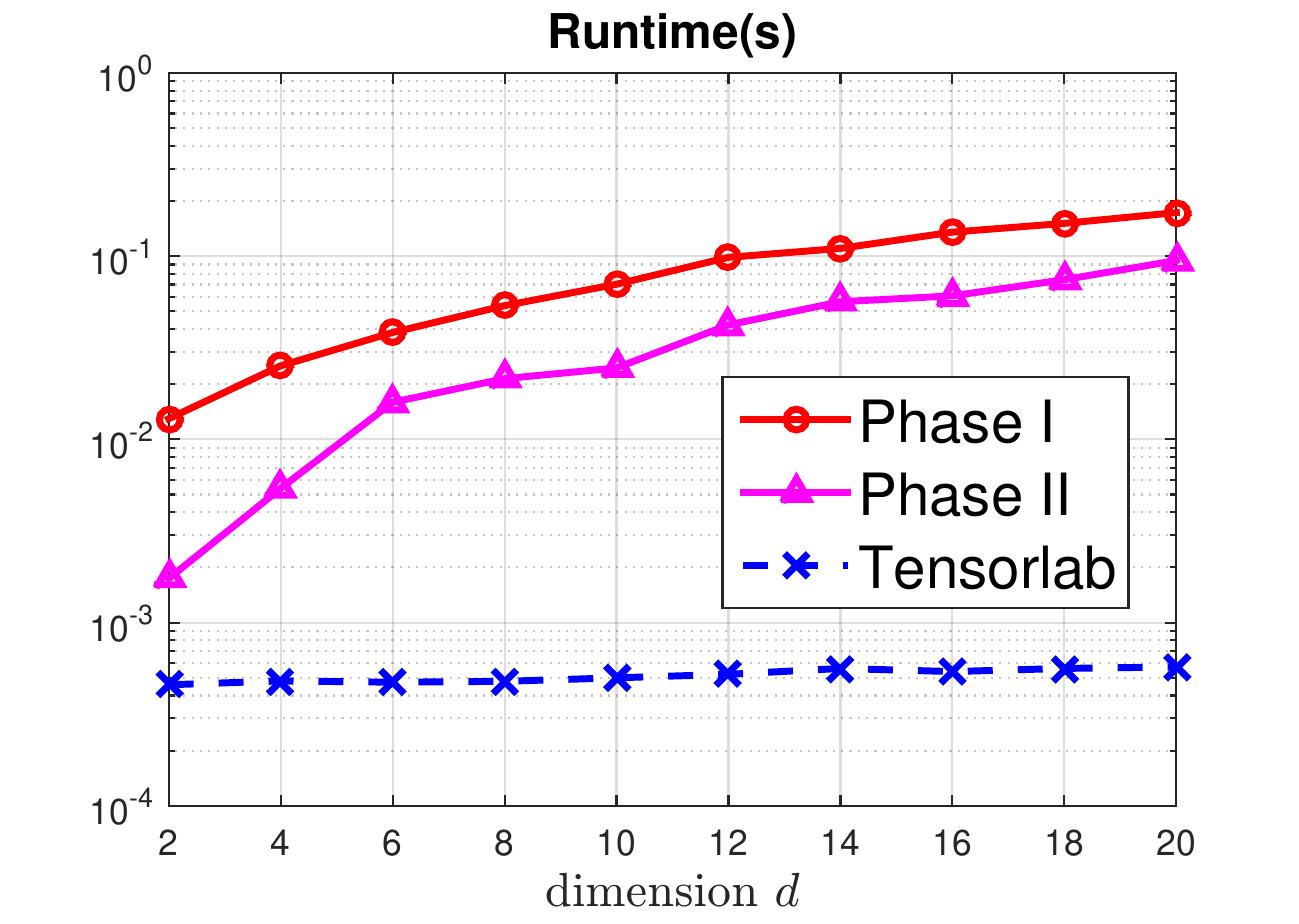}
\end{minipage}
\begin{minipage}{.499\textwidth}
\centering
     \includegraphics[width=66mm]{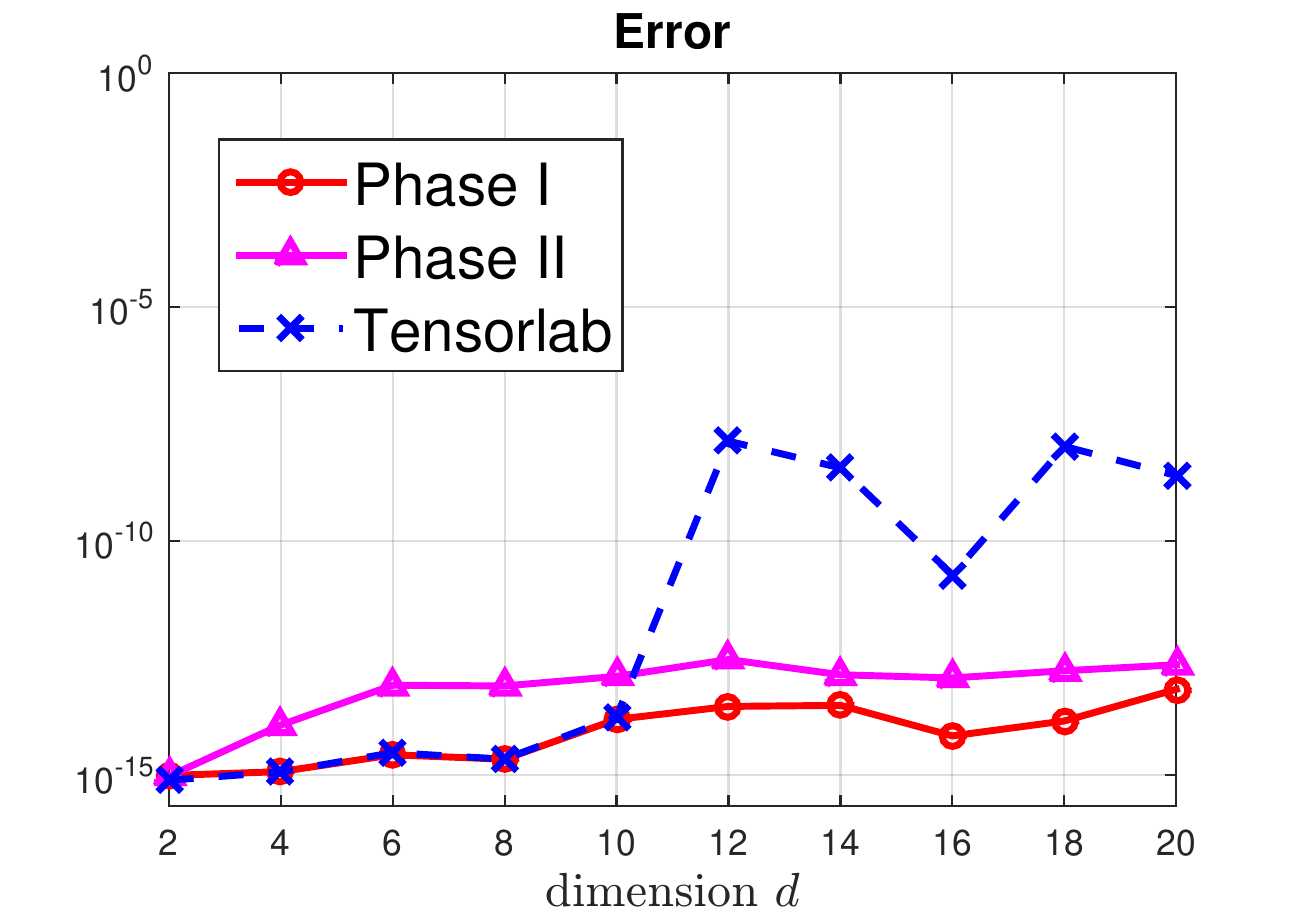}
\end{minipage}
  \caption{
Rank-1 basis matrices $r=1$, 
Fixed $m = n=10$, varying $d$ between $2$ and $20$. 
Our Algorithm gives better accuracy  when $d> n$. 
}
  \label{fig:tensorex2}
\end{figure}

We conclude that the CP-based algorithm is recommended if 
(i) the basis matrices are known to be of rank one, and 
(ii) the dimension is lower than $\min(m,n)$. 
Our algorithm, on the other hand, is slower, but substantially different in that it does not need to know that a rank-one basis exist, but will detect (in Phase I) and recover it automatically. Also it seems indifferent to linear dependent factors in the CP model.




\subsection{Quality of the convergence estimate}\label{sec:exconv}
In Section~\ref{sec:convanal} we analyzed the convergence of Algorithm~\ref{alg:getmats} in Phase II and showed that, when the error term is randomly distributed the convergence factor would be roughly $\sqrt{\frac{r}{n}}$, recall 
the remark after
Theorem~\ref{th: local convergence}. 
Here we illustrate with experiments how accurate this estimate is. 

In Figure~\ref{fig:convex} we plot a typical convergence of 
$\| \mathcal{T}_r(X) - X \|_F$
 as the iterations proceed. We generated test problems 
(randomly as before)
varying $n$ on the left ($n=10,100$) and varying $r$ on the right ($r=2,10$). 
The dashed lines indicate the convergence estimate $(\sqrt{\frac{r}{n}})^{\ell}$ after the $\ell$th iteration. Observe that in both cases the estimated convergence factors reflect the actual convergence reasonably well, and 
in particular we verify the qualitative tendency that 
 (i) for fixed matrix size $n$, 
the convergence is slower for larger rank $r$, and 
(ii) for fixed rank $r$, the convergence is faster for larger $n$.

\begin{figure}[htbp]
\begin{minipage}{.499\textwidth}
\centering
     \includegraphics[width=65mm]{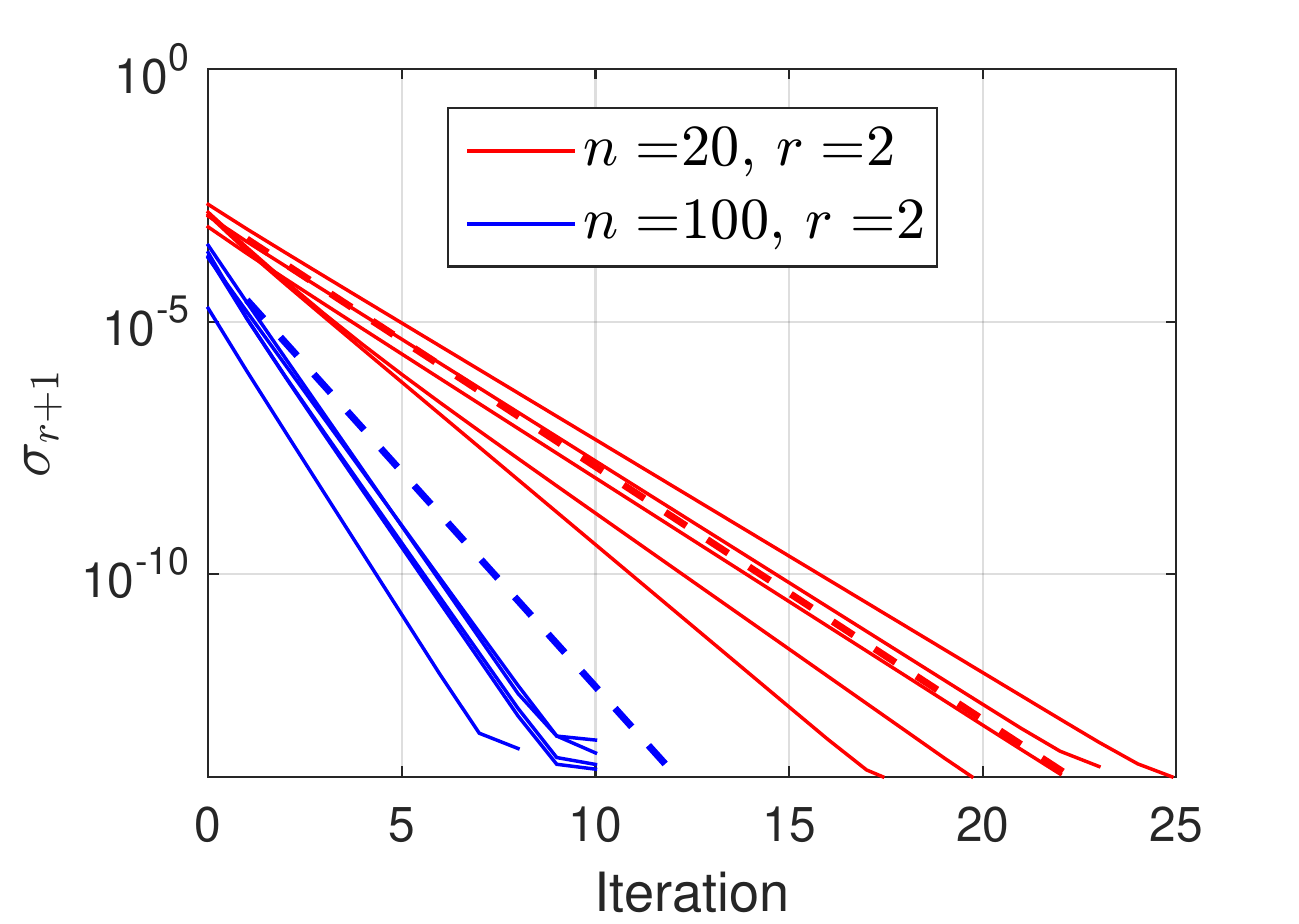}
\end{minipage}
\quad 
\begin{minipage}{.499\textwidth}
\centering
     \includegraphics[width=65mm]{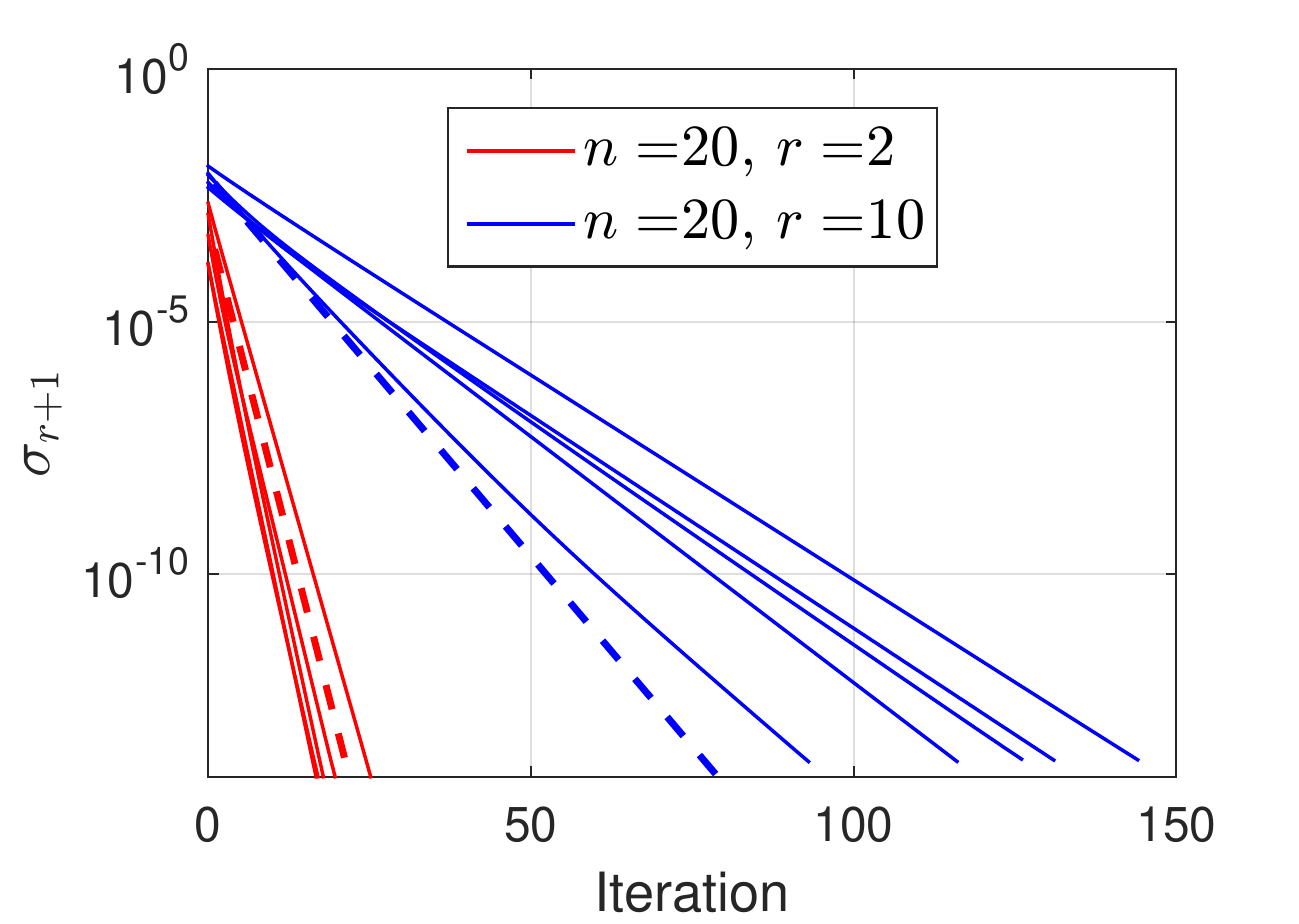}
\end{minipage}
  \caption{
Convergence of $\| \mathcal{T}_r(X) - X \|_F$ as the iterations proceed in Phase II. The convergence factor is faster for larger matrices when the rank is fixed (left), and slower for higher rank when the matrix size is fixed (right), reflecting Theorem~\ref{th: local convergence}. 
}
  \label{fig:convex}
\end{figure}


\subsection{Image separation}\label{sec:images}
One well known use of the SVD is for data and image compression, although currently it is no longer used for the JPEG format or other modern image formats. 
It is known that most images can be compressed significantly without losing the visible quality by using a low-rank approximation of the matrix that represents the image. 

Since each grayscale image can be expressed as a matrix, here we 
apply our algorithm to a set of 
four incomprehensible  images (shown as ``mixed'' in Figure~\ref{fig:andrepic}) that are random linear combinations of four `original' low-rank images. The latter were obtained from truncating four pictures (the famous `Lena',
along with photos of Audrey Hepburn, 
Angelina Jolie and Arnold Schwarzenegger, 
taken from the labeled faces in the wild dataset~\cite{LFWTech}) to rank exactly $15$ using singular value decomposition. 
As shown in Figure~\ref{fig:andrepic}, we can recover these low-rank images from the mixed ones using our Algorithm. In this experiment we had to decrease the minimal threshold in Phase I to $\tau_{tol} = 5\cdot 10^{-4}$ for obtaining the correct rank guess 15 for all four images. The standard choice $\tau_{tol} = 10^{-3}$ underestimated the target rank as 14, which resulted in poorer approximations in the subspace after Phase II (since it does not contain a rank 14 matrix), that is, the gap after singular value number 15 was less pronounced than in Fig.~\ref{fig:andrepic}. Nevertheless, the visual quality of the recovered images was equally good. 

We also note that 
in this experiment, and only in this experiment, 
restarting (Algorithm~\ref{alg:restart}) was invoked several times in Phase II, on average about 2 or 3 times. 


Of course, separation problems like this one have been considered extensively in the image processing literature~\cite{abolghasemi2012blind,bell1995information,zhao2013joint}, which is known as \emph{image separation}, and we make no claims regarding the actual usefulness of our approach in image processing. This example is included simply for an illustrative visualization of the recovery of the low-rank matrix basis by Algorithm~\ref{alg:heuristic greedy algorithm}.  In particular, our approach would not work well if the matrices of the images are not low-rank but have gradually decaying singular values. This is the case here for the original images from the database, which are not of low rank. Without the initial truncation our algorithm did not work in this example.

\begin{figure}[htbp]
  \centering
     \includegraphics[width=\textwidth]{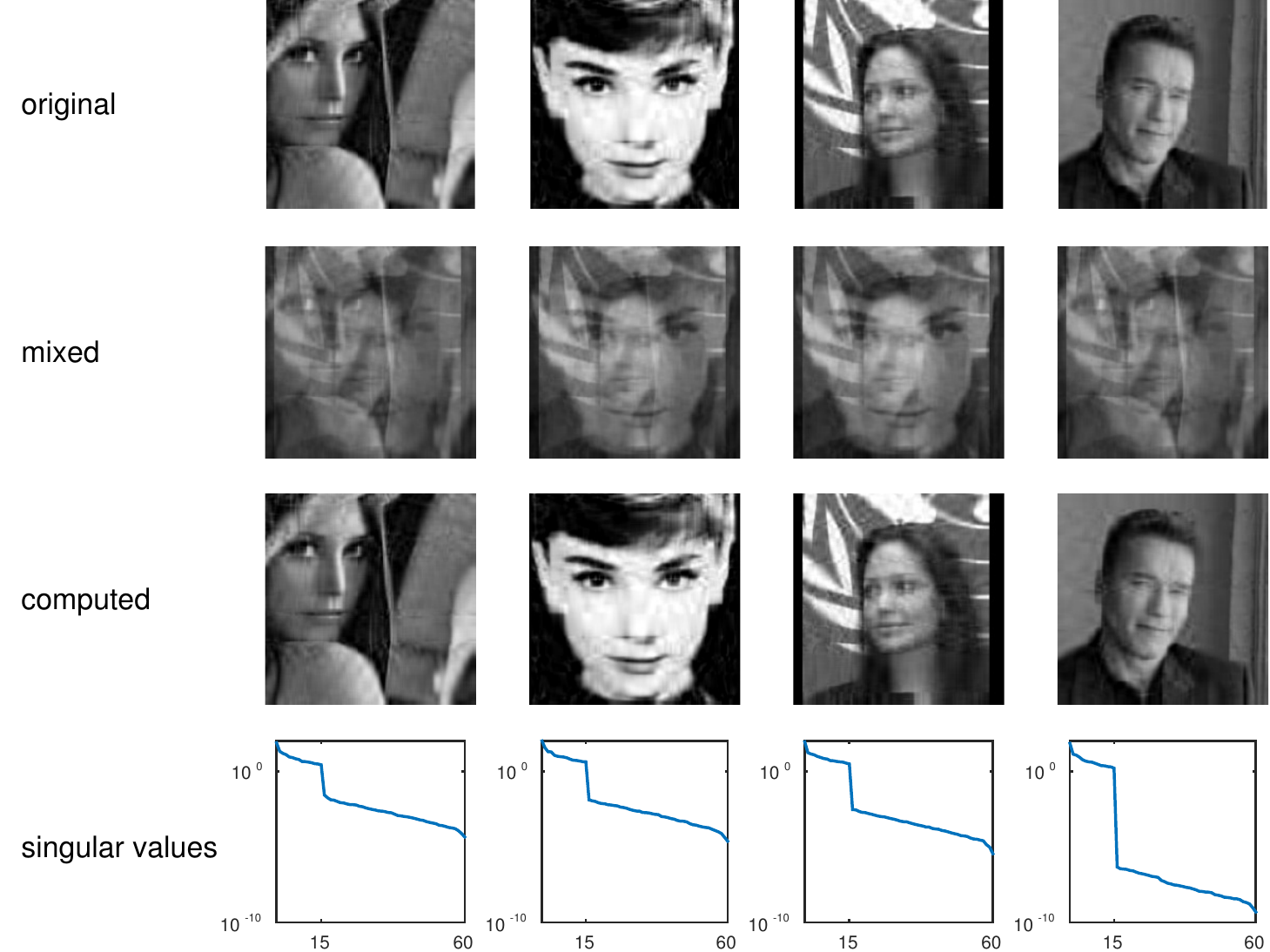}
  \caption{
We start with the middle images, which are 
obtained as random linear combinations of the rank-$15$ images (of size $200 \times 200$) in the top row. 
We apply our algorithm to obtain the four images in the third row lying in the same subspace (they were sorted accordingly). The fourth row shows the singular values of the recovered images.
}
  \label{fig:andrepic}
\end{figure}

\subsection{Computing exact eigenvectors of a multiple eigenvalue}\label{sec:exeig}
Eigenvectors of a multiple eigenvalue are not unique. For example, the identity matrix $I$ has any vector as an eigenvector. 
However, among the many possibilities one might naturally wish to obtain ``nice'' eigenvectors: for example, the columns of $I$ might be considered a good set of ``nice'' eigenvectors for $I$, as they require minimum storage. 

Numerically, the situation is even more complicated: a numerically stable algorithm computes eigenpairs $(\hat \lambda,\hat x)$ with residual $A\hat x-\hat\lambda\hat  x = O(u\|A\|)$, where $u$ is the unit roundoff. Since the eigenvector condition number is $O(\frac{1}{gap})$~\cite[Sec.~1.3]{stewart2} where $gap$ is the distance between $\lambda$ and the rest of the eigenvalues, the accuracy  of a computed eigenvector is typically $O(\frac{u\|A\|}{gap})$. This indicates the difficulty (or impossibility in general) of computing accurate eigenvectors for near-multiple eigenvalues in finite precision arithmetic. 
The common compromise is to compute a subspace corresponding to a cluster of eigenvalues, which is stable provided the cluster is well separated from the rest~\cite[Sec.~4.8]{baitemplates}. 

Here we shall show nonetheless  that it is sometimes possible to compute exact eigenvectors of  (near) multiple eigenvalues, if additional 
property is present that the eigenvectors are low-rank when matricized. 
As we discussed in the introduction, this also lets us compress the memory to store the information. Below we illustrate how this can be done with examples. 

\subsubsection{Eigenvectors of a multiple eigenvalue of a circulant matrix}\label{sec:fftmat}
As is well known, 
the eigenvector matrix of a circulant matrix is the FFT matrix~\cite[Sec.~4.8]{golubbook4th}. One can easily verify that each column of an $n^2\times n^2$ FFT matrix $F$ is rank-one when matricized to $n\times n$, exemplifying a strong low-rank property. 

Let us consider a circulant matrix $A\in\mathbb{C}^{n^2\times n^2}$ defined by 
\begin{equation}  \label{eq:ffta}
A = \frac{1}{n^2}F \Lambda F^*, 
\end{equation}
where $\Lambda = \mbox{diag}(\lambda_1,\ldots,\lambda_{n^2})$. 

Suppose for the moment that one is oblivious of the circulant structure 
(or perhaps more realistically, we can think of a matrix $A$ that is not circulant but has $d$ eigenvectors consisting of columns of $F$; such a matrix gives similar results)
and attempts to compute the $d$ smallest eigenvalues of $A$ by a standard algorithm such as QR. 

For the reason explained above, the numerically computed eigenvectors $\hat x_i$ obtained by MATLAB's {\tt eig} have poor accuracy. 
For concreteness suppose that $\lambda_1=\lambda_2=\dots =\lambda_d$ 
and $\lambda_{d+i} = \lambda_{d+i-1}+1$ for integers $i$, and we look for the eigenvectors corresponding to the first $d$ eigenvalues. 
With $n=10$ and $d=5$, 
the smallest angle between $\hat x_i$ and the first $d$ columns of the Fourier matrix were $O(1)$ for each $i$ (it should be 0 if $\hat x_i$ was exact). 
Nonetheless, the subspace spanned by the $d$ computed eigenvectors 
$[\hat x_1,\ldots, \hat x_d]$ has accuracy $O(u)$, as there is sufficient gap between $\lambda_k$ and $\lambda_{d+1}$. 
We therefore run our algorithm with the $n\times n$ matrix subspace 
\[
\mathcal{M} = \spann\{\mat(\hat x_1),\dots,\mat(\hat x_d)\}. 
\]
Our algorithm correctly  finds the rank ($=1$), and finds the 
eigenvectors $[x_1,\ldots, x_d]$, each of which is numerically rank-one and has $O(u)$ angle with a column of the Fourier matrix. 
This is an example where 
by exploiting structure we achieve high accuracy that is otherwise impossible with a backward stable algorithm; another established example being the singular values for bidiagonal matrices~\cite[Ch.~5]{demmelbook}.

For example, we let $n^2=20^2$ and compute the smallest $5$ eigenvalues of a circulant matrix $A = \frac{1}{n}F\mbox{diag}(1+\epsilon_1,1+\epsilon_2,1+\epsilon_3,1+\epsilon_4,1+\epsilon_5,6,\ldots,n^2)F^*$ where $\epsilon_i = O(10^{-10})$ was taken randomly. The matrix $A$ therefore has a cluster of five eigenvalues near $1$. The ``exact'' eigenvectors are the first five columns of the FFT matrix. 

\begin{table}[htbp]
\begin{center}
\caption{
Accuracy (middle columns) and memory usage for computed eigenvectors of a $20^2\times 20^2$ circulant matrix. 
}
\label{tab:eigvec}
\begin{tabular}{c|ccccc|c}
  &$v_1$&$v_2$&$v_3$&$v_4$&$v_5$& memory\\\hline
{\tt eig}
&4.2e-01&1.2e+00&1.4e+00&1.4e+00&1.5e+00
 & $O(n^2)$\\
{\tt eig}+Alg.~\ref{alg:heuristic greedy algorithm}
& 1.2e-12  & 1.2e-12  & 1.2e-12  & 1.2e-12 &  2.7e-14
& $O(n)$
\end{tabular}
\end{center}
\end{table}

Note that our algorithm recovers the exact eigenvector of a near-multiple eigenvalue with accuracy $O(10^{-12})$. Furthermore, the storage required to store the eigenvectors has been reduced from $5n^2$ to $5n$. 


\subsubsection{Matrices with low-rank eigenvectors}\label{sec:lowrankmat}
Of course, not every 
vector has low-rank structure when matricized. 
Nonetheless, we have observed that in many applications, the eigenvectors indeed have a low-rank structure that can be exploited. 
This observation may lead to the ability to deal with problems of scale much larger than previously possible.

Circulant matrices are an important example, as we have seen above (which clearly includes symmetric tridiagonal, symmetric banded, etc). We have observed that a sparse perturbation of a circulant matrix also has such structure. 

Other examples come from graph Laplacians. We have numerically observed that typically the Laplacian matrix of the following graphs have eigenvectors (corresponding to the smallest nonzero eigenvalues) that are low-rank: binary tree, cycle, path graph and the wheel graph all have rank 3 irrelevant of the size, the lollipop graph has rank 4 (regardless of the ratio of the complete/path parts), and the ladder graph has rank 2 and circular ladder (rank 2) regardless of the size, and barbell always has rank 5. 
Clearly, not every graph has such structure: a counterexample is a complete graph. Our empirical observation is that sparse graphs tend to have low-rank structure in the eigenvectors. 

Note that the low-rankness of the eigenvector depends also on the ordering of the vertices of the graph.\footnote{We thank Yuichi Yoshida for this observation.}
An ordering that seemed natural 
 have often exhibited low-rank property. 

Our algorithm does not need to know a priori that a low-rank structure is present, as its phase I attempts to identify whether a low-rank basis exists. 
We suspect that identifying and exploiting such structure will lead to significant improvement in both accuracy and efficiency (both in speed and memory). 
Identifying the conditions under which such low-rank structure is present is left as an open problem. 
We expect and hope that the low-rank matrix basis problem will find use in applications beyond those described in this paper. 

%
%

\appendix 
\section{Finding rank-one bases via tensor decomposition}\label{sec: reuction to tensors}
In this appendix, 
we describe the rank-one basis problem as a tensor decomposition problem. Recall that in this problem, we are promised that the given subspace $\mathcal{M}$ is spanned by rank-one matrices.
Thus we can apply Algorithm \ref{alg:getmats} (Phase II) with the precise rank guess directly. Alternatively, we can also stop after Algorithm \ref{alg:rankest} (Phase I), which in practice performs well (see Section \ref{sec: synthetic averaged}). The following tensor decomposition viewpoint leads to further algorithms.

Let $M_1,\dots,M_d$ be an arbitrary basis of $\mathcal{M}$, and let $\mathcal{T}$ be the $m\times n\times d$ tensor whose 3-slices are $M_1,\dots,M_d$. The fact that $\mathcal{M}$ possesses a rank-one basis is \emph{equivalent} to the existence of $d$ (and not less) triplets of vectors $(\mathbf{a}_\ell^{}, \mathbf{b}_\ell^{}, \mathbf{c}_\ell^{})$ 
where $\mathbf{a}_\ell^{}\in \mathbb{R}^{m}, \mathbf{b}_\ell^{}\in \mathbb{R}^{n},  \mathbf{c}_\ell^{}\in \mathbb{R}^{d}, $
such that 
\begin{equation}\label{eq:CP slicewise}
 M_k = \sum_{\ell = 1}^d c_{k,\ell}^{} \mathbf{a}_\ell^{} \mathbf{b}_\ell^T, \quad k = 1,\dots,d
\end{equation}
(here $c_{k,\ell}$ denotes the $k$th entry of $\mathbf{c}_\ell$). Namely, if such triplets 
$(\mathbf{a}_\ell^{}, \mathbf{b}_\ell^{}, \mathbf{c}_\ell^{})$ exist, then the assumed linear independence of the $M_k$ automatically implies 
that rank-one matrices $ \mathbf{a}_\ell^{} \mathbf{b}_\ell^T$ belong to $\mathcal{M}$. 
Using the outer product of vectors (denoted by $\circ$), we may express this relation in terms of the tensor $\mathcal{T}$ as
\begin{equation}\label{eq:CP decomposition}
    \mathcal{T} = \sum_{\ell=1}^d \mathbf{a}_\ell \circ \mathbf{b}_\ell \circ \mathbf{c}_\ell.
\end{equation}
This type of tensor decomposition into a sum of outer products is called the \emph{CP decomposition}, and is due to Hitchcock~\cite{Hitchcock1927a} (although the term CP decomposition appeared later). In general, the smallest $d$ required for a representation of the form~\eqref{eq:CP decomposition} is called the (canonical) rank of the tensor $\mathcal{T}$. We refer to~\cite{KoldaBader2009} and references therein for more details. In summary, we have the following trivial conclusion.

\begin{proposition}
The $d$-dimensional matrix space $\mathcal{M} = \myspan(M_1,\dots,M_d)$ possesses a rank-one basis if and only if the tensor $\mathcal{T}$ whose $3$-slices are the $M_1,\dots,M_d$ has (canonical) rank $d$. Any CP decomposition~\eqref{eq:CP decomposition} of $\mathcal{T}$ provides a rank-one basis $\mathbf{a}_1^{} \mathbf{b}_1^T, \dots, \mathbf{a}_d^{} \mathbf{b}_d^T$ of $\mathcal{M}$.
\end{proposition}

%

We remark that computing the rank of a general third-order tensor is known to be NP-hard~\cite{hastad1990,hillar2013most}. Therefore, it is NP-hard to check whether a matrix space $\mathcal{M}$ admits a rank-one basis. Nevertheless, we might try to find a rank-one basis by trying to calculate a CP decomposition~\eqref{eq:CP decomposition} from linearly independent $M_1,\dots,M_d$. We outline two common algorithms.


\subsection{Simultaneous diagonalization}\label{sec: simultaneous diagonalization}
If the tensor $\mathcal{T} \in \mathbb{R}^{m \times n \times r}$ is known to have rank $d$ and $d\leq \min(m,n)$, it is ``generically'' possible to find a CP decomposition~\eqref{eq:CP decomposition} in polynomial time using simultaneous diagonalization~\cite{Lathauwer2006,Lathauwer2004,Leurgans1993}.

Let us introduce the \emph{factor matrices} $\mathbf{A} = [\mathbf{a}_1,\ldots,\mathbf{a}_d] \in\mathbb{R}^{m\times d}$, $\mathbf{B} = [\mathbf{b}_d,\ldots,\mathbf{b}_d] \in\mathbb{R}^{n\times d}$, and $\mathbf{C} = [\mathbf{c}_1,\ldots,\mathbf{c}_d] \in\mathbb{R}^{d\times d}$.
Then~\eqref{eq:CP slicewise} reads
\[
 M_k = \mathbf{A} D_k \mathbf{B}^T, \quad k=1,\dots,d,
\]
where $D_k = \diag(\mathbf{c}_k^T)$, in which  $\mathbf{c}_k^T$ denotes the $k$th row of $\mathbf{C}$. In other words, a rank-one basis exists, if the $M_1,\dots,M_d$ can be simultaneously diagonalized. The basic idea of the algorithm of Leurgans, Ross, and Abel in~\cite{Leurgans1993} is as follows. One assumes $\rank(\mathbf{A}) = \rank(\mathbf{B}) = d$. Pick a pair $(k,\ell)$, and assume that $D_k$ and $D_\ell$ are invertible, and that $D_k^{} D_\ell^{-1}$ has $d$ distinct diagonal entries. Then it holds
\[
 M_k^{} M_\ell^+ \mathbf{A} = \mathbf{A} D_k^{} \mathbf{B}^T (\mathbf{B}^T)^+ D_\ell^{-1} \mathbf{A}^+ \mathbf{A} = \mathbf{A} D_k^{} D_\ell^{-1},
\]
where superscript $^+$ denotes the Moore-Penrose inverse. In other words, $\mathbf{A}$ contains the eigenvectors of $M_k^{} M_\ell^+$ to distinct eigenvalues, and is essentially uniquely determined (up to scaling and permutation of the columns). Alternatively, for more numerical reliability, one can compute an eigenvalue decompositions of a linear combination of all $M_k^{} M_\ell^+$ instead, assuming that the corresponding linear combination of $D_k^{} D_\ell^{-1}$ has distinct diagonal entries. Similarly, $\mathbf{B}$ can be obtained from an eigendecomposition, e.g. of $M_k^T (M_\ell^T)^+$ or linear combinations. Finally,
\[
 D_k = \mathbf{A}^+ M_k (\mathbf{B}^T)^+, \quad k = 1,\dots,d,
\]
which gives $\mathbf{C}$. The algorithm requires the construction of Moore-Penrose inverses of matrices whose larger dimension is at most $\max(m,n)$. Hence, 
the complexity is $O(mn^2)$.

The condition that the $D_k^{} D_\ell^{-1}$ or a linear combination of them should have distinct diagonal entries is not very critical, since it holds generically, if the matrices $M_1,\dots,M_d$ are randomly drawn from $\mathcal{M}$, or, when 
this is not possible, are replaced by random linear combination of themselves. The condition $\rank(\mathbf{A}) = \rank(\mathbf{B}) = d$ on the other hand, is a rather strong assumption on the rank-one basis $\mathbf{a}^{}_1\mathbf{b}_1^T,\dots,\mathbf{a}_d^{} \mathbf{b}_d^T$. It implies uniqueness of the basis, and restricts the applicability of the outlined algorithm a priori to dimension $d \le \min(m,n)$ of $\mathcal{M}$. There is an interesting implication on the condition~\eqref{eq: nontangential intersection} that we used for the local convergence proof of our algorithms. Theorem~\ref{th: local convergence} and Corollary~\ref{cor: local convergence Phase I} therefore apply at every basis element $\mathbf{a}_\ell^{} \mathbf{b}_\ell^T$ in this setting.

\begin{proposition}\label{prop: relation of full rank factor matrices to tangential interestcion}
Let $\mathbf{a}^{}_1\mathbf{b}_1^T,\dots,\mathbf{a}_d^{} \mathbf{b}_d^T$ be a rank-one basis such that $\rank(\mathbf{A}) = \rank(\mathbf{B}) = d$. Then~\eqref{eq: nontangential intersection} holds at any basis element $X_* = X_\ell = \mathbf{a}_\ell^{} \mathbf{b}_\ell^T$.
\end{proposition}

\begin{proof}
This follows immediately from Lemma~\ref{lem: suff. condition for local convergence} by taking $\tilde{\mathcal{M}} = \spann \{ \mathbf{a}_k^{} \mathbf{b}_k^T \vcentcolon k \neq \ell \}$.
\qed\end{proof}



De Lathauwer~\cite{Lathauwer2006} developed the idea of simultaneous diagonalization further. His algorithm requires the matrix $\mathbf{C}$ to have full column rank, which in our case is always true as $\mathbf{C}$ must contain the basis coefficients for $d$ linearly independent elements $M_1,\dots,M_d$. The conditions on the full column rank of $\mathbf{A}$ and $\mathbf{B}$ can then be replaced by some weaker conditions, 
but, simply speaking, too many linear dependencies in $\mathbf{A}$ and $\mathbf{B}$ will still lead to a failure. A naive implementation of De Lathauwer's algorithm in~\cite{Lathauwer2006} seems to require $O(n^6)$ operations (assuming $m=n$). 

Further progress on finding the CP decomposition algebraically under even milder assumptions has been made recently in~\cite{domanov2014canonical}. It is partially based on the following observation: denoting by $m_\ell=\mbox{vec}(M_\ell)$ the $n^2\times 1$ vectorization of $M_\ell$ (which stacks the column on top of each other), 
and defining $\mbox{Matr}(\mathcal{T}) = [m_1,\ldots, m_r]\in\mathbb{R}^{mn\times d}$, we have 
\begin{equation}  \label{eq:mivec}
\mbox{Matr}(\mathcal{T}) =  (\mathbf{B} \odot \mathbf{A} ) \mathbf{C}^T, 
\end{equation}
where
\(
\mathbf{B} \odot \mathbf{A} = [\mathbf{a}_1 \otimes \mathbf{b}_1,\ldots,\mathbf{a}_d \otimes \mathbf{b}_d] \in \mathbb{R}^{mn \times d}
\)
is the so called \emph{Khatri-Rao product} of $\mathbf{B}$ and $\mathbf{A}$ (here $\otimes$ is the ordinary Kronecker product). If $\mathbf{C}$ (which is of full rank in our scenario) would be known, then $\mathbf{A}$ and $\mathbf{B}$ can be retrieved from the fact that the matricizations of the columns of
\(
\mbox{Matr}(\mathcal{T})\mathbf{C}^{-T} = \mathbf{B} \odot \mathbf{A}
\)
must be rank-one matrices. In~\cite{domanov2014canonical} algebraic procedures are proposed that find the matrix $\mathbf{C}$ from $\mathcal{T}$.

Either way, by computing the CP decomposition for $\mathcal{T}$ we can, at least in practice, recover the rank one basis $\{\mathbf{a}_\ell^{}\mathbf{b}_\ell^T\}$ in polynomial time if we know it exists. This is verified in our MATLAB experiments using Tensorlab's \texttt{cpd} in Section~\ref{sec:comparetensor}.

\subsection{Alternating least squares}

An alternative and cheap workaround are optimization algorithms to calculate an \emph{approximate} CP decomposition of a given third-order tensor, a notable example being \emph{alternating least squares} (ALS), which was developed in statistics along with the CP model for data analysis~\cite{Carroll1970,Harshman1970}. In practice, they often work astonishingly well when the exact rank is provided.

Assuming the existence of a rank-one basis, that is, $\rank(\mathcal{T}) = d$, the basic ALS algorithm is equivalent to a block coordinate descent method applied to the function
\[
f(\mathbf{A},\mathbf{B},\mathbf{C}) = \frac{1}{2} \left\| \mathcal{T} - \sum_{\ell=1}^d \mathbf{a}_\ell \circ \mathbf{b}_\ell \circ \mathbf{c}_\ell \right\|^2_F.
\]
The name of the algorithm comes from the fact that a block update consists in solving a least squares problem for one of the matrices $\mathbf{A}$, $\mathbf{B}$ or $\mathbf{C}$, since $f$ is quadratic with respect to each of them. It is easy to derive the explicit formulas. For instance, fixing $\mathbf{A}$ and $\mathbf{B}$, an optimal $\mathbf{C}$ with minimal Frobenius norm is found from~\eqref{eq:mivec} as
\(
\mathbf{C} = \mbox{Matr}(\mathcal{T})^T (\mathbf{B}^T \odot \mathbf{A}^T)^+.
\)
The updates for the other blocks look similar when using appropriate reshapes of $\mathcal{T}$ into a matrix; the formulas can be found in~\cite{KoldaBader2009}.



The question of convergence of ALS is very delicate, and has been subject to many studies. As it is typical for these block coordinate type optimization methods for nonconvex functions, convergence can, if at all, ensured only to critical points, but regularization might be necessary, see~\cite{uschmajew2012local,LiKindermannNavasca2013,Mohlenkamp2013,XuYin2013,WangChu2014,Uschmajew2014} 
 for some recent studies, and~\cite{KoldaBader2009} in general. Practical implementations are typically a bit more sophisticated than the simple version outlined above, for instance the columns of every factor matrix should be rescaled during the process for more numerical stability. Also a good initialization of $\mathbf{A}$, $\mathbf{B}$, and $\mathbf{C}$ can be crucial for the performance. For instance one may take the leading HOSVD vectors~\cite{de2000multilinear} or the result of other methods~\cite{Kindermann2013} as a starting guess for ALS.



\bibliographystyle{spmpsci_nodoi}      
\bibliography{lowrankbases}   

%
%

\appendix

\end{document}